\documentclass[graybox]{amsart}

\usepackage{amssymb,a4wide,amsthm,amsmath,amsfonts,epsfig}
\usepackage{mathrsfs}
\usepackage{hyperref}
\input colordvi
\usepackage[active]{srcltx}
\usepackage{comment}
\theoremstyle{plain}
\newtheorem{theorem}{Theorem}[section]
\theoremstyle{remark}
\newtheorem{remark}[theorem]{Remark}

\theoremstyle{plain}
\newtheorem{corollary}[theorem]{Corollary}
\newtheorem{lemma}[theorem]{Lemma}
\newtheorem{proposition}[theorem]{Proposition}
\newtheorem{convention}[theorem]{Convention}
\newtheorem{definition}[theorem]{Definition}

\numberwithin{equation}{section}

\begin{document}
\title[Stochastic geodesic equation in $T\Bbb S^2$]{Ergodicity for a stochastic geodesic equation in the tangent bundle of the 2D sphere}
\author[L. Ba\v nas, Z. Brze\'zniak, M. Neklyudov, M. Ondrej\'at, A. Prohl]{L. Ba\v nas, Z. Brze\'zniak, M. Neklyudov, M. Ondrej\'at, A. Prohl}
\address{Fakult\"at f\"ur Mathematik, Universit\"at Bielefeld, 33501 Bielefeld, Germany}
\email{banas@math.uni-bielefeld.de}
\address{Department of Mathematics, The University of York, Heslington, York YO10 5DD, United Kingdom}
\email{zb500@york.ac.uk}
\address{Faculty of Science Office, Level 2, Carslaw Building (F07), University of Sydney NSW 2006, Australia}
\email{misha.neklyudov@gmail.com}
\address{Institute of Information Theory and Automation, Pod Vod\'arenskou v\v e\v z\'{\i} 4, CZ-182 08, Praha 8, Czech Republic, phone: ++ 420 266 052 284, fax: ++ 420 286 890 378}
\email{ondrejat@utia.cas.cz}
\address{Mathematisches Institut, Universit\"at T\"ubingen, Auf der Morgenstelle 10, D-72076 T\"ubingen, Germany}
\email{prohl@na.uni-tuebingen.de}
\thanks{The research of the fourth named author was supported by the GA \v CR Grant P201/10/0752.}
\maketitle
\begin{abstract}
We study ergodic properties of stochastic geometric wave equations on a particular model with the target being the $2D$ sphere while considering the space variable-independent solutions only. This simplification leads to a degenerate stochastic equation in the tangent bundle of the $2D$ sphere. Studying this equation, we prove existence and non-uniqueness of invariant probability measures for the original problem and we obtain also results on attractivity towards an invariant measure. We also present a suitable numerical scheme for approximating the solutions subject to a sphere constraint.
\end{abstract}

\section{Introduction}

Wave equations subject to random excitations have been largely studied in last forty years for its applications in physics, relativistic quantum mechanics or oceanography, see e.g. \cite{Cabana_1972}, \cite{Carm+Nual_1988}, \cite{Carm+Nual_1988_b}, \cite{Chow_2002}, \cite{Dal+Lev_2004}, \cite{Marcus+M_1991}, \cite{Masl+S_1993}, \cite{Mill+M_2001}, \cite{o4}, \cite{o1}, \cite{Pesz+Zab_1997}, \cite{Dal-Fr_1998}, \cite{Karcz+Z_1998}, \cite{Karcz+Z_1999}, \cite{Mill+SS_1999}, \cite{Pesz_2002}, \cite{Pesz+Zab_2000}. The mathematical research has paid attention predominantly to stochastic wave equations whose solutions took values in Euclidean spaces, however many physical theories and models in modern physics such as harmonic gauges in general relativity, non-linear $\sigma$-models in particle systems, electro-vacuum Einstein equations or Yang-Mills field theory require the target space of the solutions to be a Riemannian
manifold see e.g. \cite{Gin+Vel_1982} and \cite{Shatah+Struwe_1998}. Stochastic wave equations with values in Riemannian manifolds were first studied in \cite{Brz+O_2007} (see also  \cite{Brz+O_2009})  where existence and uniqueness of global strong solutions were proved for equations defined on the one-dimensional Minkowski space $\mathbb{R}^{1+1}$ and arbitrary Riemannian manifold. Later, in \cite{Brz+O_2013}, global existence was proved for equations on a general Minkowski space $\mathbb{R}^{1+d}$ with the target space being restricted to homogeneous spaces (for instance, a sphere) and, in \cite{Brz+O_2011}, global existence of weak solutions was proved for equations on $\mathbb{R}^{1+1}$ with an arbitrary target. The last two works admitted rougher noises than in \cite{Brz+O_2007}, but for the price of not dealing with the question of uniqueness and of worse spatial regularity of the solutions.

In the present paper, we intend to open the door and enter into the study of ergodic properties of solutions of these equations. In particular, we are interested in existence and uniqueness (or multitude) of invariant measures of the Markov semigroup associated to solutions of a stochastic geometric equation and we also want to address the questions of ergodic properties and of the rates of convergence to an attracting law, if there is any.

This goal however seems to be fairly complicated and too ambitious to achieve at once, hence we will proceed {\it a minori ad majus} and we will study just space independent solutions of a damped stochastic geometric wave equation in the 2D sphere. This particular exemplary equation is, in our opinion, quite illustrative to understand what one can expect in the general case. In this way, the stochastic equation will reduce to a degenerate second order stochastic differential equation with values in the tangent bundle $T\Bbb S^2$. We will prove that there exist plenty of invariant measures and that the system always converges in total variation to a limit law. If we however restrict the state space to a suitable submanifold in $T\Bbb S^2$ then there exists just one unique invariant measure (the normalized surface measure on this submanifold) which attracts every initial distribution in total variation with an exponential rate.

A further goal of this paper is to construct a numerical scheme for solving a class of SDEs on manifolds - the geodesic equation on the sphere with stochastic forcing. A convergent discretisation in space and time for a similar but first order stochastic Landau-Lifshitz-Gilbert equation, which is based on finite elements, is proposed in \cite{bcp1}; this scheme guarantees the sphere constraint to hold for approximate magnetisation processes and thus inherits the Lyapunov structure of the problem. As a consequence, iterates may be shown to construct weak martingale solutions of the limiting equations. Main steps of this construction are detailed in Section~\ref{napprox2}, which requires a discrete Lagrange multiplier used in the presented algorithm for iterates to inherit the sphere constraint in a discrete setting. Overall convergence of iterates is asserted in Theorem~\ref{numscheme1} which holds for this particular SDE on the sphere, but which may also be considered as a first step to numerically approximate the stochastic geometric wave equation. Again, computational examples are provided in Section~\ref{nelubo} to illustrate the results proved in this work, and motivate further analytical studies for computationally observed long-time behaviors which lack a sound analytical understanding at this stage.

The authors wish to thank Jan Seidler for valuable discussions and for pointing out the works \cite{ich1} and \cite{ich2} to us.

\section{Notation and conventions}
If $Y$ is a topological space, we will denote by $B_b(Y)$ the space of real bounded Borel functions on $Y$, by $C_b(Y)$ the space of real bounded continuous functions on $Y$, by $\mathscr B(Y)$ the Borel $\sigma$-algebra over $Y$. We will work on a probability space $(\Omega,\mathscr F,\Bbb P)$ equipped with a filtration $(\mathscr F_t)$ such that $\mathscr F_0$ contains all $\Bbb P$-negligible sets in $\mathscr F$ and $W$ will be a standard $(\mathscr F_t)$-Wiener process. Throughout this paper, all initial conditions are assumed to be $\mathscr F_0$-measurable.

\section{The problem}

Let $M$ be a compact $m$-dimensional Riemannian manifold embedded in a Euclidean space $\Bbb R^n$. Denote by $T_pM$ the tangent space at $p\in M$, by $N_pM=(T_pM)^\perp$ the normal space at $p\in M$, by $TM=\bigcup_{p\in M}T_pM$ and $T^kM=\bigcup_{p\in M}(T_pM)^k$ the tangent bundle and the $k$-tangent bundle of $M$ resp., by $S_p:T_pM\times T_pM\to N_pM$, $p\in M$ the second fundamental form of $M$ in $\Bbb R^n$ and let $W$ be, for simplicity, a one-dimensional Wiener process. According to \cite{Brz+O_2007}, the general Cauchy problem for a stochastic geometric wave equation has the form
\begin{eqnarray}\label{sgwe1}
du_t&=&\Delta u-\sum_{i=1}^mS_u(u_{x_i},u_{x_i})+S_u(u_t,u_t)+F_u(Du)+G_u(Du)\,dW
\\
u&\in& M\label{sgweconst}
\\
(u(0),u_t(0))&\in&TM\label{sgwe2}
\end{eqnarray}
where $F$ is a drift, $G$ a diffusion and $Du=(u_t,u_{x_1},\dots,u_{x_m})$. For the equation to make sense, it is required that $F:T^{m+1}M\to TM$ and $G:T^{m+1}M\to TM$ are Borel measurable and that $F_p(X_0,\dots,X_m)$ and $G_p(X_0,\dots,X_m)$ belong to the tangent space $T_pM$ for every $p\in M$ and every $X_0,\dots,X_m\in T_pM$.

In case $M$ is the unit sphere in $\Bbb R^3$ then the second fundamental form satisfies $S_p(X,Y)=-\langle X,Y\rangle p$, so if we set $F_p(X_0,X_1,X_2)=-\frac 12X_0$, $G_p(X_0,X_1,X_2)=p\times X_0$ where\footnote{Here $a\times b=(a_2b_3-a_3b_2,a_3b_1-a_1b_3,a_1b_2-a_2b_1)$ for $a,b\in\Bbb R^3$.} the diffusion term is inspired by the diffusion terms proposed in \cite{bbp1} or \cite{nekpro} in connection with the stochastic Landau-Lifshitz-Gilbert equation for ferromagnetic and nanomagnetic models then the equation \eqref{sgwe1} with the constraints \eqref{sgweconst}, \eqref{sgwe2} has the form

\begin{equation}\label{sswe}
du_t=[\Delta u+(|\nabla u|^2-|u_t|^2)u-\frac 12u_t]\,dt+u\times u_t\,dW,\qquad|u|=1,\qquad u(0)\perp u_t(0).
\end{equation}

If we consider just space independent solutions, i.e. solutions independent of the spatial variables then \eqref{sswe} reduces to an It\^o SDE
\begin{equation}\label{seqito}
du^\prime=[-|u^\prime|^2u-\frac 12u^\prime]\,dt+(u\times u^\prime)\,dW,\qquad|u|=1,\qquad u(0)\perp u^\prime(0)
\end{equation}
or, equivalently, to a Stratonovich SDE
\begin{equation}\label{seqstrat}
du^\prime=-|u^\prime|^2u\,dt+(u\times u^\prime)\,\circ dW,\qquad|u|=1,\qquad u(0)\perp u^\prime(0)
\end{equation}
which is the stochastic geodesic equation for the unit sphere\footnote{The geodesic equation for the unit sphere has the form $u^{\prime\prime}=-|u^\prime|^2u$, $|u|=1$, $u^\prime(0)\perp u(0)$.}. Let us rewrite \eqref{seqstrat} to two equations of first order equations 
\begin{equation}\label{gsde}
dz=f(z)\,dt+g(z)\,\circ dW,\qquad z\in T\Bbb S^2,\qquad z(0)\in T\Bbb S^2
\end{equation}
where $T\Bbb S^2\subseteq\Bbb R^6$ is the tangent bundle of $\Bbb S^2$, i.e. $T\Bbb S^2=\{(u,v):|u|=1,u\perp v\}$ and
\begin{equation}\label{defvf}
z=\left(\begin{array}{c}u\\v\end{array}\right),\qquad f(z)=\left(\begin{array}{c}v\\-|v|^2u\end{array}\right),\qquad g(z)=\left(\begin{array}{c}0\\u\times v\end{array}\right).
\end{equation}

\begin{remark}\label{rest0} Observe that restrictions of $f$ and $g$ to $T\Bbb S^2$ are vector fields on the manifold $T\Bbb S^2$. Hence \eqref{gsde} is a correctly defined stochastic differential equation on the manifold $T\Bbb S^2$, cf. \cite[Chapter V]{ikwa}.
\end{remark}

The equation \eqref{seqito} and its equivalent formulations \eqref{seqstrat}, \eqref{gsde} will be {\it the object of study} of the present paper. It is also important to realize while reading the paper that \eqref{seqito} is a particular case of the stochastic geometric wave equation \eqref{sgwe1}-\eqref{sgwe2}.

\section{Basic properties of solutions of the SDE}

We will study existence of global solutions, dependence on initial conditions, some further qualitative properties of solutions of the equation \eqref{gsde} and the Feller property of the associated Markov semigroup.

\subsection{Global existence} The nonlinearities of the equation \eqref{gsde} are locally Lipschitz on $\Bbb R^6$ hence, by the standard existence result (see e.g. \cite[Lemma 2.1]{ikwa}), the equation \eqref{gsde} considered without the constraint,
\begin{equation}\label{pomoc}
dz=f(z)\,dt+g(z)\,\circ dW,\qquad z(0)\in T\Bbb S^2,
\end{equation}
has a local solution $z$ in $\Bbb R^6$ defined upto an explosion time $\tau>0$, i.e.
\begin{equation}\label{explosionc}
\limsup_{t\uparrow\tau}|z(t)|=\infty\quad\textrm{almost surely on}\quad[\tau<\infty].
\end{equation}

\begin{proposition}\label{gext} Every solution to \eqref{pomoc} is global and satisfies $z=(u,v)\in T\Bbb S^2$, i.e. it is a solution to the equation \eqref{gsde}. Moreover, $|v(t)|=|v(0)|$ for every $t\ge 0$ almost surely.
\end{proposition}

\begin{proof}
Applying the It\^o formula to $|u|^2$, we obtain that $\phi=|u|^2-1$ satisfies almost surely on $[0,\tau)$ the ODE
\begin{equation}\label{pomocr}
\phi^{\prime\prime}=-2|v|^2\phi-\frac 12\phi^\prime,\qquad\phi(0)=0,\qquad\phi^\prime(0)=0.
\end{equation}
Hence, by the uniqueness of the solutions to the equation \eqref{pomocr}, we obtain that $\phi=0$ on $[0,\tau)$, consequently, $|u|=1$ on $[0,\tau)$ almost surely. In particular, differentiating $|u|^2=1$, we obtain that $u\perp v=0$ on $[0,\tau)$ almost surely. Now, applying the It\^o formula to $|v|^2$, we obtain that $\varphi=|v|^2$ satisfies on $[0,\tau)$ almost surely the equation
$$
\varphi^\prime=-(1+2\langle u,v\rangle)|v|^2+|u\times v|^2.
$$
The right hand side equals to 
$$
-(1+2\langle u,v\rangle)|v|^2+|u|^2|v|^2-\langle u,v\rangle^2=0
$$
as $u\perp v$ and $|u|=1$ almost surely. Hence $|v|$ is pathwise constant. In particular, $\tau=\infty$ almost surely by \eqref{explosionc}.
\end{proof}

\subsection{The Markov and the Feller property}
Define $Y=\Bbb R^n$. It is well known that if $\tilde f$, $\tilde g$ are $C^\infty$ vector fields on $\Bbb R^n$ with a compact support and $u^\xi$ denotes the solution of the equation 
\begin{equation}\label{lipcoe}
dX=\tilde f(X)\,dt+\tilde g(X)\circ dW,\qquad X(0)=\xi
\end{equation}
for an $\mathscr F_0$-measurable $Y$-valued random variable $\xi$ then the solutions of the equation \eqref{lipcoe} satisfy the Markov property and define  a Feller semigroup\footnote{We allow here a little inaccuracy. More precisely, the semigroup is defined on the space of bounded Borel functions on $Y$.} on $Y$ by which we mean that
\begin{itemize}
\item[(a)] the transition function
$$
q_{t,x}(A)=\Bbb P\,[u^x(t)\in A],\qquad t\ge 0,\,\,x\in Y,\,\, A\in\mathscr B(Y)
$$
is jointly measurable in $(t,x)\in[0,\infty)\times Y$ for every $A\in\mathscr B(Y)$,
\item[(b)] the endomorphisms on $B_b(Y)$
$$
Q_t\varphi(x)=\Bbb E\,\varphi(u^x(t)),\qquad t\ge 0,\,\,x\in Y,\,\,\varphi\in B_b(Y)
$$
satisfy the semigroup property, i.e. $Q_t\circ Q_s=Q_{t+s}$ for every $t,s\ge 0$,
\item[(c)] $Q_t\varphi$ is continuous on $Y$ whenever $t\ge 0$ and $\varphi\in C_b(Y)$,
\item[(d)] $\Bbb E\,[\varphi(u^\xi(t))|\mathscr F_s]=(Q_{t-s}\varphi)(u^\xi (s))$ holds a.s. for every $\varphi\in B_b(Y)$, $0\le s\le t$ and an $\mathscr F_0$-measurable $Y$-valued random variable $\xi$,
\end{itemize}
see e.g. \cite[Section 9.2.1]{dapzab_1992}. In fact, (a) and (c) follow simply from the fact that 
\begin{equation}\label{joicont}
Q_t\varphi(x)\textrm{ is jointly continuous in }(t,x)\textrm{ on }[0,\infty)\times Y\textrm{ if }\varphi\in C_b(Y),
\end{equation}
see again \cite[Section 9.2.1]{dapzab_1992} for the proof of \eqref{joicont}, and the semigroup property (b) follows from the Markov property (d).

Moreover, if $\varphi\in C^2(Y)$ with derivatives of order $0,1,2$ bounded then 
\begin{equation}\label{diffprop}
\rho(t,x)=Q_t\varphi(x)\textrm{ belongs to }C^{1,2}([0,\infty)\times Y)
\end{equation}
with $\rho$, $\frac{\partial\rho}{\partial t}$, $\frac{\partial\rho}{\partial x_i}$, $\frac{\partial^2\rho}{\partial x_i\partial x_j}$ bounded for every $i,j\in\{1,\dots,n\}$ and it is a solution to the backward Kolmogorov equation
\begin{equation}\label{kolm1}
\frac{\partial U}{\partial t}=\sum_{i=1}^n\tilde f_i\frac{\partial U}{\partial x_i}+\frac 12\sum_{i=1}^n\sum_{j=1}^n\tilde g_i\frac{\partial}{\partial x_i}\left(\tilde g_j\frac{\partial U}{\partial x_j}\right),\qquad U(0,x)=\varphi(x)\textrm{ for every }x\in Y
\end{equation}
unique in the class $C^{1,2}([0,\infty)\times Y)$, see e.g. \cite[Section 9.3]{dapzab_1992}.

Unfortunately, the coefficients of the equation \eqref{gsde} are not compactly supported so we cannot simply conclude that the solutions of \eqref{gsde} satisfy the Markov property and define a Feller semigroup in the sense (a)-(d) above. Yet, it is true, as will be shown below. 

\begin{definition}\label{defi} From now on, $z^\xi$ denotes the solution of \eqref{gsde} with the initial condition $\xi$, $p_{t,x}(A)=\Bbb P\,[z^x(t)\in A]$ and $P_t\varphi(x)=\Bbb E\,\varphi(z^x(t))$ are defined for $\varphi\in B_b(T\Bbb S^2)$, $t\ge 0$, $x\in T\Bbb S^2$ and $A\in\mathscr B(T\Bbb S^2)$.
\end{definition}

\begin{proposition}\label{markov} The solutions of \eqref{gsde} satisfy the Markov property and define a Feller semigroup on $T\Bbb S^2$. In fact, $P_t\varphi(x)$ is jointly continuous in $(t,x)$ on $[0,\infty)\times T\Bbb S^2$ for every $\varphi\in C_b(T\Bbb S^2)$ and
$$
\Bbb E\,[\varphi(z^\xi(t))|\mathscr F_s]=(P_{t-s}\varphi)(z^\xi (s))\textrm{ almost surely}
$$
holds for every $\varphi\in B_b(T\Bbb S^2)$, $0\le s\le t$ and every initial $T\Bbb S^2$-valued initial condition $\xi$.
\end{proposition}

\begin{proof}
Let us prove the joint continuity assertion first. Assume that $(t_n,x_n)\to(t,x)$ in $[0,\infty)\times T\Bbb S^2$ and let $\sup_n|x_n|\le l$. Let $\tilde f$, $\tilde g$ be compactly supported $C^\infty$ vector fields on $\Bbb R^6$ so that $f=\tilde f$ and $g=\tilde g$ on the ball of radius $l$ in $\Bbb R^6$. Now $|z^{x_n}(t)|=|x_n|\le l$ and $|z^x(t)|=|x|\le l$ holds for every $t\ge 0$ a.s. by Proposition \ref{gext} and hence $z^{x_n}$, $z^{x}$ are also solutions to the equation 
$$
dX=\tilde f(X)\,dt+\tilde g(X)\circ dW.
$$
So, if $\varphi\in C_b(T\Bbb S^2)$ and $\tilde\varphi\in C_b(\Bbb R^6)$ is any extension of $\varphi$ (which always exists by the Tietze theorem) then
$$
\lim_{n\to\infty}P_{t_n}\varphi(x_n)=\lim_{n\to\infty}\Bbb E\,\tilde\varphi(z_n(t_n))=\Bbb E\,\tilde\varphi(z(t))=P_t\varphi(x)
$$
by \eqref{joicont}. 

To prove the Markov property, let $\xi=(\xi^1,\xi^2)$ be a $T\Bbb S^2$-valued initial condition and define $\xi_k=(\xi^1,\xi^2\mathbf 1_{[|\xi^2|\le k]})$. Then $\xi_k$ take values in $T\Bbb S^2$ and by Proposition \ref{gext}, $|z^{\xi_k}(t)|=|\xi_k|\le\sqrt{1+k^2}$. Let $\tilde f$, $\tilde g$ be compactly supported $C^\infty$ vector fields on $\Bbb R^6$ so that $f=\tilde f$ and $g=\tilde g$ on the ball of radius $\sqrt{1+k^2}$ in $\Bbb R^6$ and define $Q_t\phi(y)=\Bbb E\,\phi(u^y(t))$ for $\phi\in B_b(\Bbb R^6)$, $y\in\Bbb R^6$, $t\ge 0$ and $u^y$ the solutions to $dX=\tilde f(X)\,dt+\tilde g(X)\circ dW$, $X(0)=y$. By the first part of the proof, we know that $P_t\varphi(x)=Q_t\tilde\varphi(x)$ holds for every $x\in T\Bbb S^2$ such that $|x|\le\sqrt{1+k^2}$, $\varphi\in B_b(T\Bbb S^2)$, $\tilde\varphi\in B_b(\Bbb R^6)$, $\varphi=\tilde\varphi$ on $T\Bbb S^2$ and $t\ge 0$. 

Now $z^{\xi_k}=u^{\xi_k}$ and if we define $A_k=[|\xi^2|\le k]$ and $\tilde\varphi\in B_b(\Bbb R^6)$ extends $\varphi\in B_b(T\Bbb S^2)$ then
$$
\mathbf 1_{A_k}\Bbb E\,[\varphi(z^\xi(t))|\mathscr F_s]=\Bbb E\,[\mathbf 1_{A_k}\varphi(z^\xi(t))|\mathscr F_s]=\Bbb E\,[\mathbf 1_{A_k}\varphi(z^{\xi_k}(t))|\mathscr F_s]=\mathbf 1_{A_k}\Bbb E\,[\varphi(z^{\xi_k}(t))|\mathscr F_s]=
$$
$$
\mathbf 1_{A_k}\Bbb E\,[\tilde\varphi(u^{\xi_k}(t))|\mathscr F_s]=\mathbf 1_{A_k}(Q_{t-s}\tilde\varphi)(u^{\xi_k}(s))=\mathbf 1_{A_k}(P_{t-s}\varphi)(z^{\xi_k}(s))=\mathbf 1_{A_k}(P_{t-s}\varphi)(z^{\xi}(s))\textrm{ a.s.}
$$
by the Markov property of solutions of the equation \eqref{lipcoe}. To obtain the result, let $k\to\infty$.
\end{proof}

\section{Multitude of invariant measures}

Now we are ready to prove that the equation \eqref{gsde} and, consequently, also the equation \eqref{sswe} have many invariant measures due to the geometric nature of the equation.

\begin{definition}\label{basicdef} Let $Y$ be a Polish space, $r_{t,x}(\cdot)$ probability measures on $\mathscr B(Y)$ indexed by $(t,x)\in[0,\infty)\times Y$ such that $r_{t,x}(A)$ is jointly measurable in $(t,x)$ on $[0,\infty)\times Y$ for every $A\in\mathscr B(Y)$ and the operators 
$$
R_t\varphi(x)=\int_Y\varphi\,dr_{t,x},\qquad \varphi\in B_b(Y),\,\,t\ge 0
$$
satisfy the semigroup property on $B_b(Y)$. We introduce the adjoint endomorphisms $R^*_t$ acting on the space of probability measures on $\mathscr B(Y)$
$$
R^*_t\nu(A)=\int_Yr_{t,x}(A)\,d\nu(x),\qquad t\ge 0,\,\,A\in\mathscr B(Y).
$$
A probability measure $\nu$ on $\mathscr B(Y)$ is called {\it invariant} provided that
$$
R^*_t\nu=\nu\mbox{ for all }t\ge 0\textrm{ and }A\in\mathscr B(Y).
$$
A probability measure on $\mathscr B(Y)$ is called {\it ergodic} provided that it is an extreme point in the convex set of invariant probability measures.
\end{definition}

\begin{remark} To make the meaning of the above definition clear, apply the Markov property in Proposition \ref{markov} with  $s=0$. If $\xi$ is an $\mathscr F_0$-measurable $T\Bbb S^2$-valued random variable with a distribution $\nu$ then $P^*_t\nu$ is the law of $z^\xi(t)$.
\end{remark}

At this moment, we introduce subsets of the tangent bundle $T\Bbb S^2$
\begin{equation}\label{submanem}
M_r=\{(u,v)\in T\Bbb S^2:|v|=r\},\qquad r\ge 0.
\end{equation}

\begin{remark}[Invariance] If $r>0$ and $x\in M_r$ then $z^x(t)\in M_r$ for every $t\ge 0$ almost surely. If $|u|=1$ then $z^{(u,0)}(t)=(u,0)$ for every $t\ge 0$ almost surely. These conclusions follow directly from Proposition \ref{gext}.
\end{remark}

\begin{corollary} Let $r>0$. For every $t\ge 0$, $P_t$ is an endomorphism on $B_b(M_r)$.
\end{corollary}

\begin{corollary}\label{cordelt} Let $x\in M_0$. Then $\delta_x$ is an invariant measure.
\end{corollary}

We are going to prove that there is more to see, than what was disclosed by Corollary \ref{cordelt}, on the sets $M_r$ as far as invariant measures are concerned.

\begin{remark}\label{rest1} Observe that, for every $r>0$, the mappings $f$ and $g$ in \eqref{defvf} are vector fields on the manifold $M_r$.
\end{remark}

In view of Remark \ref{rest1}, we can introduce the following second order differential operator on $M_r$.

\begin{definition} Define the second order differential operator
\begin{equation}\label{defopera}
\mathcal A\varphi=f(\varphi)+\frac 12g(g(\varphi))
\end{equation}
for $\varphi\in C^2(M_r)$ for $r>0$.
\end{definition}

The following result follows from Theorem 3.1 in \cite{ikwa} but we found out that, in this case, it is easier to give a direct prove rather than to check the assumptions in Section 3 in \cite{ikwa}.

\begin{proposition}\label{kolm3} Let $r>0$ and let $\varphi\in C^2(M_r)$. Then $\rho(t,x)=P_t\varphi(x)$ belongs to $C^{1,2}([0,\infty)\times M_r)$ and satisfies the backward Kolmogorov equation
\begin{equation}\label{kolm2}
\frac{\partial\rho}{\partial t}=\mathcal A\rho\quad\textrm{on}\quad[0,\infty)\times M_r,\qquad\rho(0,\cdot)=\varphi.
\end{equation}
On the other hand, if $\rho\in C^{1,2}([0,\infty)\times M_r)$ satisfies \eqref{kolm2} then $\rho(t,x)=P_t\varphi(x)$ on $[0,\infty)\times M_r$.
\end{proposition}

\begin{proof}
Let $k\in\Bbb N$ and let $\tilde f$ and $\tilde g$ be $C^\infty$ vector fields on $\Bbb R^6$ such that $\tilde f=f$ and $\tilde g=g$ on the centered ball in $\Bbb R^6$ of radius $R=\sqrt{1+r^2}$. Denote by $u^x$ the solution of $dX=\tilde f(X)\,dt+\tilde g(X)\circ dW$, $X(0)=x$ and let $Q_t$ be the associated Markov operators. Let $\tilde\varphi\in C^2(\Bbb R^6)$ be a compactly supported extension of $\varphi$. Then $z^x=u^x$ for every $x\in M_r$ by Proposition \ref{gext}, $J(t,x)=Q_t\tilde\varphi(x)\in C^{1,2}([0,\infty)\times\Bbb R^6)$ by \eqref{diffprop}, hence $J(t,x)=\rho(t,x)$ for $(t,x)\in[0,\infty)\times M_r$. In particular, $\rho\in C^{1,2}([0,\infty)\times M_r)$ and \eqref{kolm2} holds by \eqref{kolm1}.

To prove the converse assertion, extend $\rho$ to a function in $C^{1,2}([0,\infty)\times\Bbb R^6)$, let $t>0$ and apply the It\^o formula to $\rho(t-r,z^x(r))$ for $r\in[0,t]$, obtaining
$$
\varphi(z^x(t))=\rho(0,z^x(t))=\rho(t,x)+\int_0^tg(\rho)(t-r,z^x(r))\,dW.
$$
Taking expectations on both sides yields the claim.
\end{proof}

The next assertion is obvious if $Q\in\Bbb R^3\otimes\Bbb R^3$ is a unitary matrix with $\operatorname{det}Q=1$ due to the invariance of the equation \eqref{gsde} for positively oriented unitary matrices. But it also holds if $\operatorname{det}Q=-1$. To prove this, we are going to use the uniqueness of the solutions of the backward Kolmogorov equation.

\begin{corollary}\label{rotations} Let $Q$ be a $3\times 3$-unitary matrix. Denote by $\widetilde Q=\operatorname{diag}\,[Q,Q]\in\Bbb R^6\otimes\Bbb R^6$. Then 
$$
p(t,\widetilde Qx,A)=p(t,x,[\widetilde Q\in A])
$$
holds for every $(t,x)\in[0,\infty)\times M_r$, every $A\in\mathscr B(M_r)$ and every $r>0$.
\end{corollary}

\begin{proof} Let $\varphi\in C^2(M_r)$ and define $\rho(t,x)=P_t\varphi(x)$ for $(t,x)\in[0,\infty)\times M_r$. Then $\rho$ verifies \eqref{kolm2}. Now define $\varrho(t,x)=\rho(t,\widetilde Qx)$ for $(t,x)\in[0,\infty)\times M_r$ which we can do since $\widetilde Q$ is a diffeomorphism on $M_r$. Then $\varrho\in C^{1,2}([0,\infty)\times M_r)$ and
$$
\frac{\partial\varrho}{\partial t}(t,x)-\mathcal A\varrho(t,x)=\frac{\partial\rho}{\partial t}(t,\widetilde Qx)-\mathcal A\rho(t,\widetilde Qx)=0\quad\textrm{on}\quad[0,\infty)\times M_r,\qquad\varrho(0,\cdot)=\varphi(\widetilde Q\cdot).
$$
So, form the uniqueness part of Proposition \ref{kolm3}, we obtain that 
\begin{equation}\label{thetran}
P_t\varphi(\widetilde Qx)=P_t(\varphi\circ\widetilde Q)(x)\quad\textrm{on}\quad [0,\infty)\times M_r.
\end{equation}
By density of $C^2(M_r)$ in $C(M_r)$ we get that \eqref{thetran} holds for every $\varphi\in C(M_r)$ and consequently for every $\varphi\in B_b(M_r)$.
\end{proof}

Now we are ready to describe some analytic properties of the Markov semigroup $(P_t)$ on $M_r$.

\begin{theorem}\label{semigrprop} Let $r>0$. Then $(P_t)$ is a $C_0$-semigroup on $C(M_r)$, $P_t[C^2(M_r)]\subseteq C^2(M_r)$, $C^2(M_r)$ is contained in the domain of the infinitesimal generator $A$ of $(P_t)$ and $A=\mathcal A$ on $C^2(M_r)$.
\end{theorem}

\begin{proof}
The $C_0$ property follows from the joint continuity in Proposition \ref{markov} and the invariance of $C^2(M_r)$ under $P_t$ from Proposition \ref{kolm3}. By the It\^o formula, 
$$
P_t\varphi(x)=\varphi(x)+\int_0^tP_s(\mathcal A\varphi)(x)\,ds,\qquad t\ge 0,\,\, x\in M_r,
$$
so $\varphi$ belongs to the domain of the infinitesimal generator $A$ of $(P_t)$ and $A\varphi=\mathcal A\varphi$.
\end{proof}

\begin{corollary}\label{invmr} Let $r>0$. Then there exists an invariant measure with support in $M_r$.
\end{corollary}

\begin{proof}
Let $\theta$ be a Borel probability measure with support in $M_r$. The semigroup $(P_t)$ is Feller on $B_b(T\Bbb S^2)$, the average probability measures $\frac{1}{T}\int_0^TP^*_s\theta\,ds$ are supported in $M_r$, hence they are tight and therefore any of its weak cluster points is an invariant probability measure according to the Krylov-Bogolyubov theorem, see e.g. Corollary 3.1.2 in \cite{daza}.
\end{proof}

We have proved so far that the tangent bundle $T\Bbb S^2$ decomposes to invariant sets 
$$
T\Bbb S^2=\bigcup_{x\in M_0}\{x\}\cup\bigcup_{r>0}M_r
$$
where on each if these sets there exists an invariant measure.

\section{Numerical simulations}\label{numsoll}

In this section, we present a numerical algorithm for approximating the solutions of \eqref{gsde} and consequent simulations that lead us to conjecture that $(P^*_t)$ restricted to $M_r$ attracts every initial distribution on $M_r$ to the normalized surface measure on $M_r$. In particular, this would mean that the normalized surface measure on $M_r$ is the unique invariant measure on $M_r$, cf. Corollary \ref{invmr}.

\subsection{Numerical approximation}\label{napprox2}
We propose a non-dissipative, symmetric discretization of \eqref{seqito} to construct strong solutions and numerically study long-time asymptotics. Let $\{ (U^n, V^n)\}_{n}$ be
approximate iterates of $\{(u(t_n), \dot{u}(t_n))\}_n$ on an equi-distant mesh $I_k$ of size $k>0$,
covering $[0,T]$. We denote $d_t \varphi^{n+1} := \frac{1}{k} (\varphi^{n+1} - \varphi^n)$. \\

\textbf{Algorithm.} Let $(U^0, V^0) := \bigl(u_0, \dot{u}(0)\bigr)$, and $U^{-1} := U^0 - k V^0$.
For $n \geq 0$, find $(U^{n+1}, V^{n+1}, \lambda^{n+1}) \in {\mathbb R}^{3+3+1}$,
such that for $\Delta W_{n+1} := W(t_{n+1}) - W(t_n) \sim {\mathscr N}(0,k)$ holds

\begin{eqnarray}\nonumber
V^{n+1}-V^n &=& k \frac{\lambda^{n+1}}{2} (U^{n+1} + U^{n-1}) + \frac{1}{4}(U^{n+1} + U^{n-1}) \times
(V^{n+1} + V^{n}) \Delta W_{n+1} \\ \label{forma1}
d_t U^{n+1} &=& V^{n+1} \\ \nonumber
 \lambda^{n+1} &=&
\left\{ \begin{array}{ll}
0 & \mbox{ for } \frac{1}{2}(U^{n+1} + U^{n-1}) = 0\,, \\
- \frac{(V^n, V^{n-1})}{\vert \frac{1}{2}(U^{n+1} + U^{n-1})\vert^2}
 & \mbox{ for } \frac{1}{2}(U^{n+1} + U^{n-1}) \neq 0 \mbox{ and } n \geq 1\,,  \\
 - \frac{(V^0, V^{1}) - \frac{1}{2} \vert V^0\vert^2}{\vert \frac{1}{2}(U^{n+1} + U^{n-1})\vert^2}
 & \mbox{ for } \frac{1}{2}(U^{1} + U^{-1}) \neq 0 \mbox{ and } n = 0\,.
\end{array}  \right.
\end{eqnarray}
The choice of the Lagrange multiplier $\lambda^{n+1}$ ensures that $\vert U^{n+1}\vert = 1$ for $n \geq 0$; the
case $n=0$ has to compensate for the fact that the defined $U^{-1}$ is not
necessarily of unit length.

To see the formula (\ref{forma1})$_3$ for $n \geq 1$, we multiply
(\ref{forma1}) with $\frac{1}{2}(U^{n+1} + U^{n-1})$ and use the discrete product formula
$$ (d_t \varphi^n, \psi^n) = - (\varphi^{n-1}, d_t \psi^n) + d_t (\varphi^n, \psi^n)$$
to find
$$ \frac{1}{2} (d_t V^{n+1}, U^{n+1} + U^{n-1}) = -\frac{1}{2} (V^n, V^{n+1}+V^{n-1}) +
\frac{1}{2} d_t \bigl( V^{n+1}, U^{n+1} + U^n - U^n + U^{n-1}\bigr)\, ,$$
where $(\cdot, \cdot)$ denotes the scalar product in ${\mathbb R}^3$, and $\vert \cdot \vert =
(\cdot, \cdot)^{1/2}$.
Since $(V^{n+1}, U^{n+1} + U^{n}) = 0$, we further obtain
\begin{eqnarray*}
&=& -\frac{1}{2} \Bigl( (V^n, V^{n+1} + V^{n-1}) + k d_t(V^{n+1},-V^n)\Bigr) \\
&=& -\frac{1}{2} \Bigl( (V^n, V^{n+1} + V^{n-1}) - (V^{n+1},V^n) + (V^n, V^{n-1})\Bigr) = -(V^n, V^{n-1})\, .
\end{eqnarray*}
Hence $-(V^n, V^{n-1}) = \lambda^{n+1} \vert \frac{1}{2}(U^{n+1} + U^{n-1})\vert^2$, which yields the formula for
$\lambda^{n+1}$ in (\ref{forma1}).

For $n = 0$, we conclude similarly, using $\langle U^0, V^0\rangle = 0$, and the definition of
$U^{-1}$.

\begin{theorem}\label{numscheme1}
Let $T >0$, and $k \leq k_0(U^0,V^0)$ be sufficiently small. For every $n \geq 0$, there exist unique ${\mathbb R}^{3+3}$-valued random variables $(U^{n+1},V^{n+1})$ of Algorithm such that $\vert U^{n+1} \vert = 1$ for all $n \geq 0$, and 
$$ E(V^{n+1}) = E(V^0) \qquad \mbox{where } E(\varphi) = \frac{1}{2} \vert \varphi\vert^2\, .$$
Define processes $({\mathscr U}^{\pm}_{k'}, {\mathscr V}_{k'}^{\pm})$ from the iterates $\{ (U^{n+1}, V^{n+1})\}_{n\geq 0}$ according to the following prescription: for ${\mathbb R}^3$-valued iterates $\{\varphi^n\}_{n \geq 0}$ on the mesh $I_k$ that covers $[0,T]$ define for every $t \in [t_n, t_{n+1})$ functions
\begin{eqnarray*}
{\mathscr \varphi}_k(t) &:=& \frac{t-t_n}{k} \varphi^{n+1} + \frac{t_{n+1}-t}{k} \varphi^n\,, \\
{\mathscr \varphi}_k^-(t) &:=& \varphi^n\,, \qquad \mbox{and } {\mathscr \varphi}^+_k(t) := \varphi^{n+1}\, .
\end{eqnarray*}
Then $({\mathscr U}^{\pm}_{k'}, {\mathscr V}_{k'}^{\pm}) \rightarrow (u,v)$ in $C \bigl([0,T]; {\mathbb R}\bigr)$ as $k' \rightarrow 0)$ a.s. where $(u,v)$ is strong solution of (\ref{forma1}). Moreover, the stochastic forcing term exerts a damping in direction $\mathscr V^{n+1}_{k'}$.
\end{theorem}

Solvability of (\ref{forma1}) is shown by an inductional argument that is based on Brouwer's fixed point theorem: an auxiliary problem is introduced which excludes the case where $\frac{1}{2}(U^{n+1}+U^{n-1}) = 0$ when computing $\lambda^{n+1}$; for sufficiently small $k>0$, constructed solutions of the auxiliary problem are in fact solutions of (\ref{forma1}). The convergence follows from a compactness argument which is based on an energy identity while preserving the sphere constrain. The details of the proof will be omitted, cf. the related works in \cite{bbp1} and \cite{bcp1}.

\begin{proof}
{\em 1.~Auxiliary problem.} Fix $n \geq 1$. For every $0 < \epsilon \leq \frac{1}{4}$, define the continuous function ${\mathcal F}^{\omega}_{\epsilon}: {\mathbb R}^3 \rightarrow {\mathbb R}^3$ where
\begin{equation}\label{forma2}
{\mathcal F}^{\omega}_{\epsilon}(W) := \frac{2}{k}(W-U^n) + k \frac{(V^n, V^{n-1})}{\max\{ \vert W\vert^2, \epsilon\}} W - W \times (V^n - \frac{2}{k} U^n) \Delta W_{n+1}\, .
\end{equation}
We compute respectively,
\begin{eqnarray*}
\frac{2}{k} ( W-U^n,W) &\geq& \frac{2}{k} \bigl( \vert W\vert - \vert U^n\vert \bigr) \vert W\vert\,, \\
k \frac{(V^n, V^{n-1})}{\max\{ \vert W\vert^2,\epsilon\}} \vert W \vert^2 &\geq& - k \vert V^n\vert\, \vert V^{n-1}\vert\, .
\end{eqnarray*}
Since the stochastic term in (\ref{forma2}) vanishes after multiplication with $W$, there exists some
function $R_n >0$ such that
$$ \bigl( {\mathcal F}^{\omega}_{\epsilon}(W), W \bigr) \geq 0 \qquad \forall\, W \in
\bigl\{\varphi \in {\mathbb R}^3:\, \vert \varphi \vert \geq R_n(U^n, V^n, V^{n-1})\bigr\}\, .$$
Then, Brouwer's fixed point theorem implies the existence of $W^{\star} \equiv \frac{1}{2}(U^{n+1} + U^{n-1})$, such that
${\mathcal F}^{\omega}_{\epsilon}\bigl( \frac{1}{2}(U^{n+1} + U^{n-1})\bigr) = 0$ for all $\omega \in \Omega$.

The argument easily adopts to the case $n=0$.

{\em 2.~Solvability and energy identity.} We show that $\frac{1}{2}(U^{n+1}+U^{n-1})$ solves ${\mathcal F}^{\omega}_{0} \bigl(\frac{1}{2}(U^{n+1} +
U^{n-1})\bigr) = 0$. By induction, it suffices to verify that
\begin{eqnarray}\nonumber
\vert \frac{1}{2}(U^{n+1}+U^{n-1})\vert &=& \vert \frac{k}{2} V^{n+1} + \frac{1}{2}(U^{n} + U^{n-1})\vert
\geq \vert \frac{1}{2}(U^n + U^{n-1})\vert - \frac{k}{2} \vert V^{n+1}\vert \\ \nonumber
&\geq& \vert U^{n-1}\vert - \frac{k}{2} \bigl(\vert V^n\vert + \vert V^{n+1}\vert\bigr) \\ \label{formb1}
&\geq& 1-\frac{1}{4} - \frac{k}{2} \vert V^{n+1}\vert \stackrel{!}{>} \frac{1}{2}\, ,
\end{eqnarray}
for $k \leq k_0(U^0,V^0) < 1$ sufficiently small.

Let $n \geq 1$. For all $0 \leq \ell \leq n$, there holds $\vert U^{\ell}\vert = 1$, and
\begin{equation}\label{formb2}
E(V^{\ell}) = E(V^0) \, .
\end{equation}
Then $W^{\star} = \frac{1}{2}(U^{n+1} + U^{n-1})$ from Step 1.~solves
\begin{equation}\label{forma3}
k d_t V^{n+1} = \frac{\lambda^{n+1}_{\epsilon} k }{2} (U^{n+1} + U^{n-1}) + \frac{1}{4}(U^{n+1}+U^{n-1}) \times
(V^{n+1}+V^n) \Delta W_{n+1}\,,
\end{equation}
where
$$ \lambda^{n+1}_{\epsilon} = - \frac{(V^n,V^{n-1})}{\max\{\epsilon, \vert \frac{1}{2}(U^{n+1}+
U^{n-1})\vert^2 \}}\, .$$
Testing (\ref{forma3}) with $\frac{1}{2}(U^{n+1}-U^{n-1}) = \frac{k}{2}(V^{n+1}+V^{n})$, and using binomial
formula $\frac{1}{2}(U^{n+1}+U^{n-1}, U^{n+1}-U^{n-1}) = \frac{1}{2} \bigl(\vert U^{n+1}\vert^2-1\bigr)$,
as well as $\vert U^{n+1}\vert^2 \leq k^2 \vert V^{n+1}\vert^2 + \vert U^n\vert^2$, and the
inductive assumption $\vert U^n\vert^2 = 1$,
\begin{eqnarray}\nonumber
d_t  \vert V^{n+1}\vert^2 &\leq& \frac{\vert \lambda^{n+1}_{\epsilon}\vert}{4} k^2 \vert V^{n+1}\vert^2 \\
\label{einschu1}
&\leq& \frac{\vert V^n\vert\, \vert V^{n-1}\vert\, k^2 \vert V^{n+1}\vert^2}{4\max\{ \epsilon,
(1-\frac{1}{4}-k\vert V^{n+1}\vert)^2\}} \leq \frac{k^2}{4\epsilon}
\vert V^n\vert\, \vert V^{n-1}\vert\, \vert V^{n+1}\vert^2 \,,
\end{eqnarray}
for $\epsilon \leq \frac{1}{2}$. By a (repeated) use of the discrete version of Gronwall's inequality, there
exists a constant $C>0$ independent on time $T > 0$, such that
\begin{equation}\label{einschu2}
\vert V^{n+1}\vert^2 \leq C\, \vert V^0\vert^2\, .
\end{equation}
As a consequence, (\ref{formb1}) is valid, and hence
${\mathcal F}^{\omega}_{\epsilon}\bigl( \frac{1}{2}(U^{n+1} + U^{n-1})\bigr) = 0$ for $\epsilon = 0$;
therefore,
$U^{n+1}$ solves (\ref{forma1}), satisfies the sphere constraint, and conserves the
Hamiltonian, i.e., (\ref{formb2}) holds for $0 \leq \ell \leq n+1$.

For $n=0$, we argue correspondingly, starting in (\ref{einschu1}) with
$$d_t \vert V^1 \vert^2 \leq C \lambda_\epsilon^1 \vert U^{1} + U^{-1}\vert^{2} \leq
\frac{k \vert V^0 \vert (\vert V^0\vert + \vert V^{1}\vert )^2}{\max\{ \epsilon, \vert \frac{1}{2} (U^{1} + U^{-1})\vert\} }\, ,$$
from which we again infer (\ref{einschu2}), and (\ref{formb1}).

Uniqueness of $\bigl\{ (U^n, V^n)_n\bigr\}_n$ follows by an energy argument, using (\ref{formb1}),
(\ref{formb2}), $k \leq k_0$, and the discrete version of Gronwall's inequality.

{\em 3.~Convergence.} We rewrite (\ref{forma1}) in the form
\begin{eqnarray}\nonumber
d v &=&- \vert v\vert^2 u \, dt + u \times v \, \circ d W\,, \\ \label{formb3}
d u &=& v dt \,,\\ \nonumber
(u_0,v_0)&\in&T{\mathbb S}^2\, .
\end{eqnarray}
We now show the convergence of $({\mathscr U}^{\pm}_{k'}, {\mathscr V}_{k'}^{\pm})$ to the solution. It is because of the discrete sphere constraint and the (discrete) energy identity that sequences
$$ \bigl\{ ({\mathscr U}^{\pm}_k, {\mathscr V}^{\pm}_k)\bigr\}_k \subset C \bigl([0,T]; {\mathbb R}\bigr)$$
are uniformly bounded. Moreover, there holds for all $t \geq 0$
\begin{eqnarray}\nonumber
{\mathscr V}(t) &=& {\mathscr V}(0) + \int_0^{t^+} \frac{{\mathscr \lambda}^+}{2} [{\mathscr U}^+ + {\mathscr U}^- -
k {\mathscr V}^-]\, ds \\ \label{forma4a}
&& + \frac{1}{4} \int_0^{t^+} ({\mathscr U}^+ + {\mathscr U}^- - k {\mathscr V}^-) \times ({\mathscr V}^+ + {\mathscr V}^-)\, dW(s)\, , \\ \nonumber
{\mathscr U}(t) &=& {\mathscr U}(0) + \int_0^{t^+} {\mathscr V}(s)\, ds\, .
\end{eqnarray}
Then, by (\ref{forma4a})$_2$, (\ref{formb1}), and H\"older continuity property of $W$, sequences
$\bigl\{ ({\mathscr U}_k, {\mathscr V}_k)\bigr\}_{k}$ are equi-continuous. Hence, by Arzela-Ascoli theorem, there exist sub-sequences $\bigl\{ ({\mathscr U}_{k'}, {\mathscr V}_{k'})\bigr\}_{k'}$,
and continuous processes $({\mathscr U}, {\mathscr V})$ on $[0,T]$ such that
\begin{equation}\label{formb5}
\Vert {\mathscr U}_{k'} -u\Vert_{C([0,T]; {\mathbb R}^3)} +
                           \Vert {\mathscr V}_{k'} -v\Vert_{C([0,T]; {\mathbb R}^3)} \rightarrow 0
                           \qquad (k' \rightarrow 0) \qquad {\mathbb P}-\mbox{a.s.}
\end{equation}
We identify limits in (\ref{forma4a}). The only crucial term is the stochastic (It\^o) integral term which may be
stated in the form
\begin{equation}\label{bb1}
\frac{1}{2} \int_0^{t^+} ({\mathscr U}^+ + {\mathscr U}^- - k {\mathscr V}^-) \times
({\mathscr V}^- + \frac{k}{2} \dot{\mathscr V})\, dW(s)\, .
\end{equation}
We easily find for every $t \in [0,T]$,
\begin{equation*}
\frac{1}{2} \int_0^{t^+} ({\mathscr U}^+ + {\mathscr U}^- - k {\mathscr V}^-) \times
{\mathscr V}^-\, {\rm d}W(s) \rightarrow \int_0^t u\times v\, dW(s)
\qquad (k \rightarrow 0) \qquad {\mathbb P}-\mbox{a.s.}
\end{equation*}
The remaining term in (\ref{bb1}) involves $\frac{k}{2}\dot{\mathscr V}$, which will be substituted by
identity (\ref{forma1})$_1$,
\begin{eqnarray}\nonumber
&& \frac{1}{2} \int_0^{t^+} ({\mathscr U}^+ + {\mathscr U}^- - k {\mathscr V}^-) \times \Bigl(
{\mathscr V}^- + k \frac{\lambda^+}{4} ({\mathscr U}^+ + {\mathscr U}^- - k {\mathscr V}^-) \\
\label{forma4c}
&&\qquad
+ \frac{1}{4} ({\mathscr U}^+ + {\mathscr U}^- - k {\mathscr V}^-) \times
({\mathscr V}^- + \frac{k}{2} \dot{\mathscr V}) \Delta W_{n+1}
\Bigr)\, dW(s)\, .
\end{eqnarray}
If compared to (\ref{bb1}), the critical factor $\frac{k}{2} \dot{\mathscr V}$ is now scaled by an
additional $\Delta W_{n+1}$; using again (\ref{forma1})$_1$, It\^o's isometry, and the estimate
${\mathbb E} \vert \Delta W_{n+1}\vert^{2^p} \leq Ck^{2^{p-1}}$ then lead to
\begin{eqnarray}\nonumber
&&\frac{1}{8} \int_0^{t^+} ({\mathscr U}^+ + {\mathscr U}^-) \times \Bigl(
({\mathscr U}^+ + {\mathscr U}^- - k {\mathscr V}^-) \times ({\mathscr V}^-
+ \frac{k}{2} \dot{\mathscr V})
\Bigr) (W^+ - W^-)\, dW(s) \\ \label{klaa}
&&\qquad \rightarrow \frac{1}{2} \int_0^t u \times (u \times v)\, ds \qquad
(k \rightarrow 0) \qquad {\mathbb P}-\mbox{a.s.},
\end{eqnarray}
for all $t \in [0,T]$.
As a consequence, there holds
$$v(t) = v(0) + \int_0^t \vert v\vert^2 u\, ds + \int_0^t u \times v\, dW(s) +
\frac{1}{2} \int_0^t u \times (u \times v)\, ds\,,$$
where the last term is the Stratonovich correction.
The proof is complete.\end{proof}
\begin{remark}\label{am1}
1.~Let $\vert V^0\vert$ be constant, and $({\mathscr V}^{n+1}, {\mathscr U}^{n+1}) :=
({\mathbb E} V^{n+1}, {\mathbb E} U^{n+1})$. Then
\begin{equation}\label{c0}
\Bigl\vert {\mathscr V}^{n+1} - {\mathscr V}^n - k \bigl[ {\mathbb E}\vert V^{n}\vert^2
{\mathscr U}^{n+1}  - \frac{1}{2} {\mathscr V}^{n+1}\bigr] \Bigr\vert \leq C k^2\, ,
\end{equation}
i.e. the stochastic forcing term exerts a damping in direction ${\mathscr V}^{n+1}$. To show this result,
we start with
\begin{eqnarray}\nonumber
{\mathscr V}^{n+1} - {\mathscr V}^n &=& \frac{k}{2} {\mathbb E} \bigl[ \lambda^{n+1} (U^{n+1} + U^{n-1})\bigr]
+ \frac{1}{2} {\mathbb E} \bigl[ (U^{n+1} + U^{n-1}) \times V^{n+1/2} \Delta W_{n+1} \bigr] \\ \label{coa}
&=:& I + II\, .
\end{eqnarray}
We use Theorem~\ref{numscheme1}, and an approximation argument to conclude that
\begin{eqnarray*}
I &=& - \frac{k}{2} \, {\mathbb E} \bigl[ \langle V^n, V^{n-1}\rangle \Bigl( 1 - [1- \frac{1}{\vert \frac{1}{2}(U^{n+1}
+ U^{n-1})\vert^2}]\Bigr) (U^{n+1} + U^{n-1})\bigr] \\
&= & - \frac{k}{2} \, {\mathbb E} \bigl[ \langle V^n, V^{n-1}\rangle ({\mathscr U}^{n+1} + {\mathscr U}^{n-1})\bigr]
+ {\mathcal O}(k^3) \\
&=& - k\,  {\mathbb E}  \vert V^n\vert^2 {\mathscr U}^{n+1} + {\mathcal O}(k^2)\, ,
\end{eqnarray*}
thanks to the power property of expectations, and $\bigl\vert \vert \frac{1}{2}(U^{n+1} + U^{n-1})\vert^2 -1\bigr\vert \leq Ck^2$.

We use the identity $U^{n+1} = kV^{n+1} + U^n$ for the leading term in $II$, and properties of the vector product to conclude that
\begin{equation*}
II = \frac{k}{4} {\mathbb E} \bigl[ (V^{n+1}-V^n) \times V^{n}\Delta W_{n+1}\bigr] + \frac{1}{4} {\mathbb E} \bigl[ (U^{n} + U^{n-1}) \times
(V^{n+1} - V^n) \Delta W_{n+1}\bigr] =: II_a + II_b\, .
\end{equation*}
We easily verify $\vert II_a \vert \leq Ck^2$, thanks to (\ref{forma1})$_1$, and properties of iterates given
in Theorem~\ref{numscheme1}. For $II_b$, we use (\ref{forma1})$_1$ as well, and the relevant term is then
\begin{eqnarray*}
&& \frac{1}{16} {\mathbb E} \Bigl[ (U^n + U^{n-1}) \times \bigl( (U^{n+1} + U^{n-1}) \times (V^{n+1} + V^n)\bigr)\vert \Delta_{n+1}\vert^2\Bigr] \\
 &&\qquad = \frac{1}{2} {\mathbb E} \Bigl[ U^{n} \times \bigl( U^{n} \times V^{n}\bigr)\vert \Delta_{n+1}\vert^2\Bigr] + {\mathcal O}(k^2) \\
 &&\qquad = - \frac{k}{2} {\mathscr V}^{n} + {\mathcal O}(k^2)
 = - \frac{k}{2} {\mathscr V}^{n+1} + {\mathcal O}(k^2)\, ,
 \end{eqnarray*}
thanks to the power property of expectations, earlier boundedness results of iterates $\{ (U^n, V^n)\}_n$,
the fact that $\vert U^n\vert = 1$ for all $n \geq 0$, the cross product formula $a \times (b \times c) =
b \langle a, c\rangle - c \langle a, b\rangle$, and another approximation argument. This observation then
settles (\ref{c0}).

2.~Strong solutions of \eqref{seqito} satisfy $$\vert u(t)\vert = 1\,, \qquad E \bigl(v(t)\bigr) = E(v_0) \qquad \forall\, t \in [0,T]\,,$$
and are unique, due to Lipschitz continuity of coefficients in \eqref{seqito}; hence, the whole sequence
$\bigl\{ ({\mathscr U}_k, {\mathscr V}_k)\bigr\}_k$ converges to $(u,v)$, for $k \rightarrow 0$.

3.~Increments of a Wiener process may be approximated by a sequence of general, not necessarily Gaussian random variables, which properly approximate higher moments of $\Delta W_{n+1}$;
martingale solutions of \eqref{seqito} may then be obtained by a more involved argumentation using theorems of Prohorov and Skorokhod; cf.~\cite{bcp1}.
\end{remark}

\subsection{Numerical experiments}\label{nelubo}
In this section we present some numerical obtained by the above Algorithm that has been applied to a slightly more general problem than \eqref{seqito}
$$
d\dot u=-|\dot u|^2u\,dt+ \sqrt{D}(u\times\dot u)\circ dW,
$$
where $D$ is a fixed constant that controls the intensity of the noise term. The Lagrange multiplier was computed as
\begin{equation}\label{lagr2}
\lambda^{n+1} = \frac{-(V^n,U^{n+1} + U^{n-1})+\frac{1}{2k}(1-|U^{n-1}|^2)}{\left|\frac{1}{2}(U^{n+1} + U^{n-1})\right|^2}\, .
\end{equation}
The above formula is equivalent to the corresponding expression in (\ref{forma1}). However, the present formulation (\ref{lagr2}) is slightly more convenient for numerical computations, since it ensures that the round off errors in the constraint $|U^n|=1$ do not accumulate over time. The solution of the nonlinear scheme (\ref{forma1}) is obtained up to machine accuracy by a simple fixed-point algorithm, cf. \cite{bpsch}.

The probability density function $\hat{f}^n$ was constructed  with $N=20000$ sample paths. For all computations in this section we take the time step size $k=0.001$ and the initial conditions $u(0) = (0,1,0)$, $\dot u(0) = (1,0,0)$. The initial probability density function associated with the above initial conditions is a Dirac delta function concentrated around $u(0)$.

In Figure~\ref{fig:d1_dens} we display the computed probability density $\hat{f}^n$ for $D=1$, $T=60$ at different time levels. Initially the probability density function is advected in the direction of the initial velocity and is simultaneously being diffused. For early times, the diffusion seems to act predominantly in the direction perpendicular to the initial velocity. In Figure~\ref{fig:d1_dens} we display the time averaged probability density function $\overline{f}$, the trajectory $\mathbb{E}(u(t))$ and a zoom at $\mathbb{E}(u(t))$ near the center of the sphere.

\begin{figure}[htp]
\center
\includegraphics[width=0.23\textwidth]{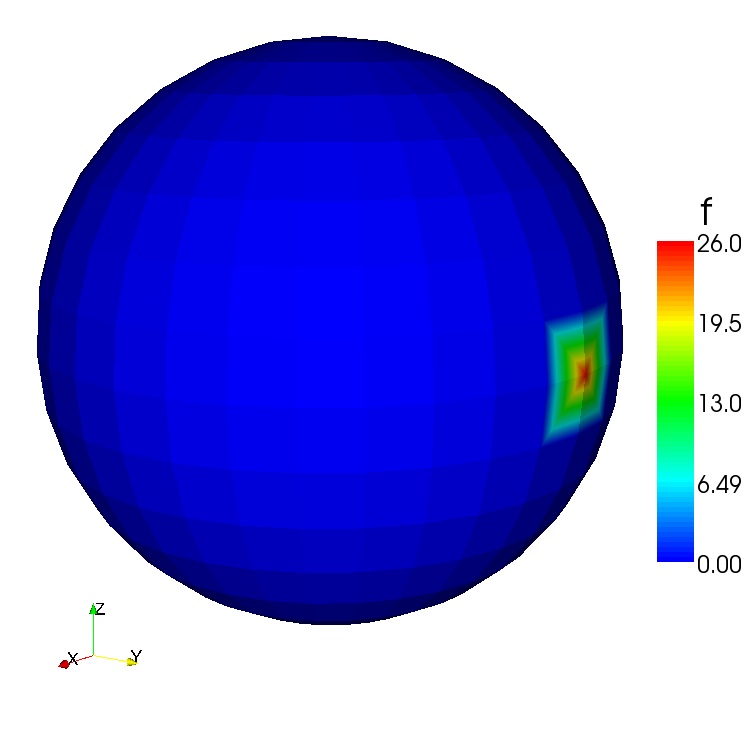}
\includegraphics[width=0.23\textwidth]{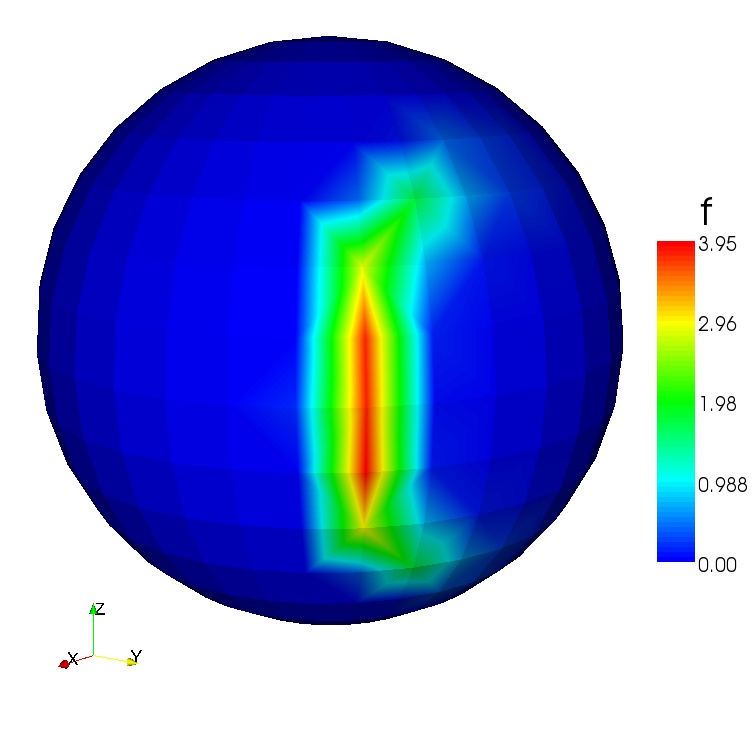}
\includegraphics[width=0.23\textwidth]{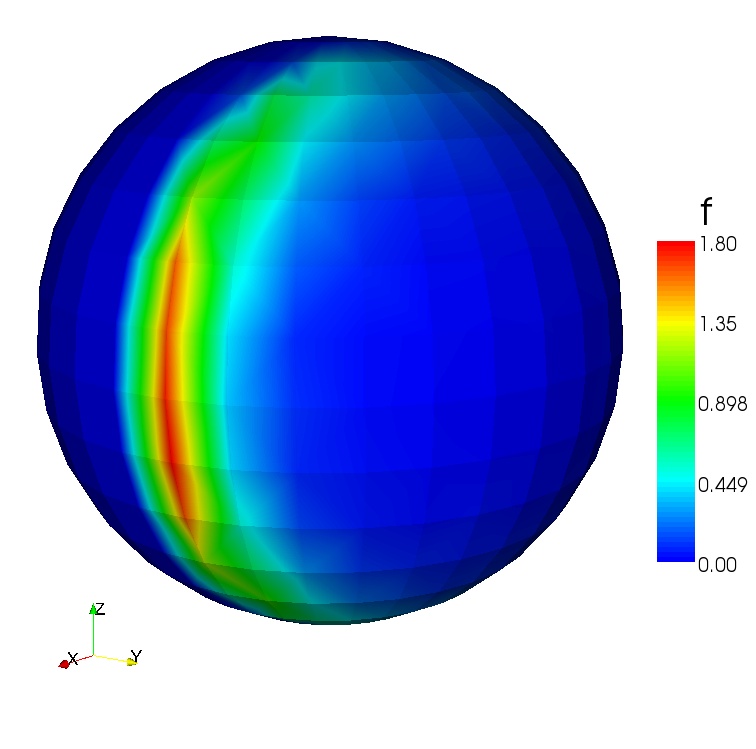}
\includegraphics[width=0.23\textwidth]{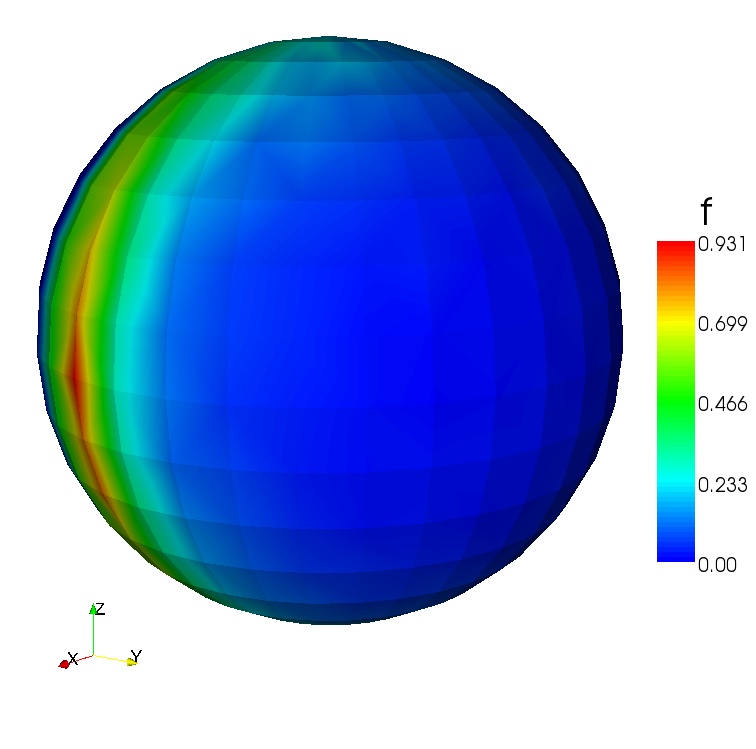}
\includegraphics[width=0.23\textwidth]{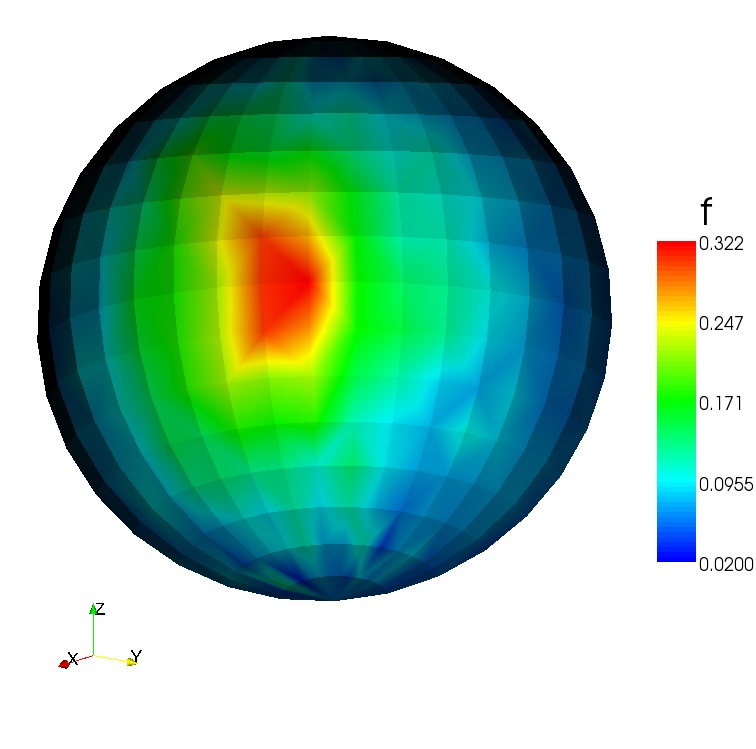}
\includegraphics[width=0.23\textwidth]{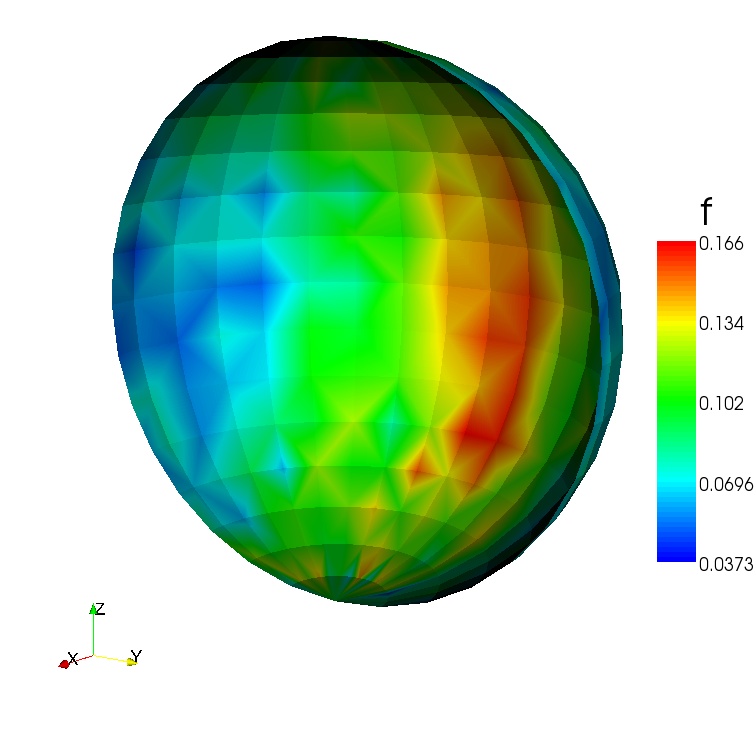}
\includegraphics[width=0.23\textwidth]{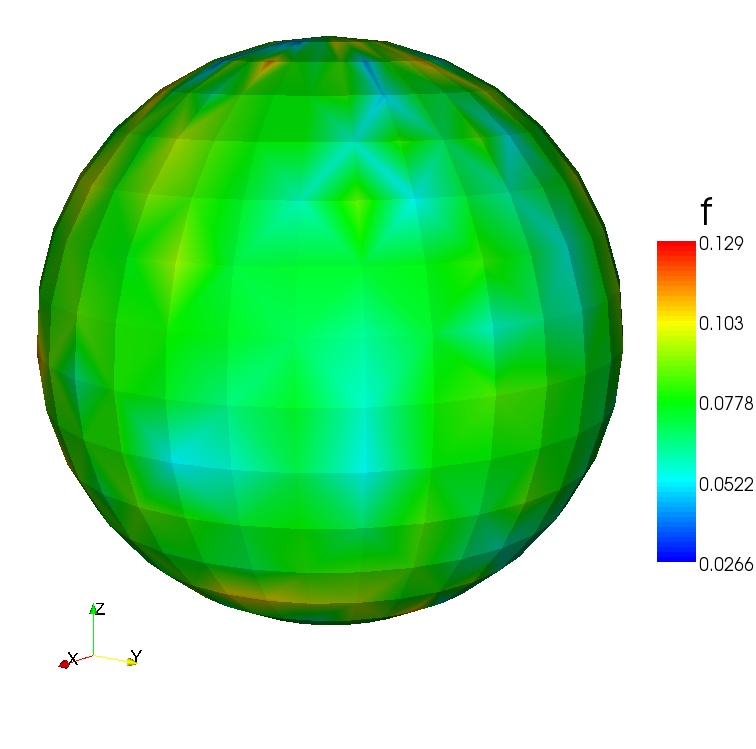}
\includegraphics[width=0.23\textwidth]{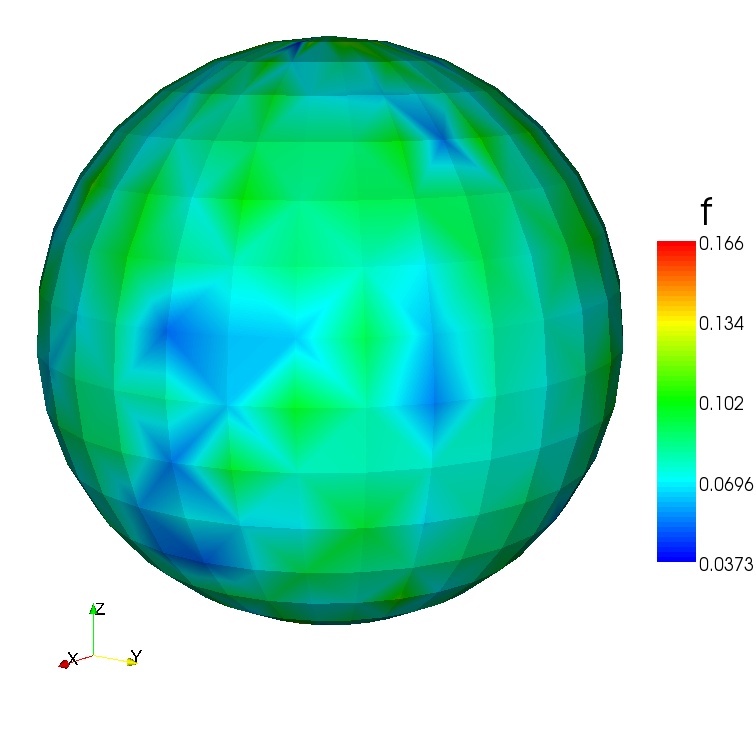}
\caption{Approximate probability density of $u$ for $D=1$ at $t=0,1,1.5,2.1,4.3,5.5,10,60$.}
\label{fig:d1_dens}
\end{figure}

\begin{figure}[htp]
\center
\includegraphics[width=0.3\textwidth]{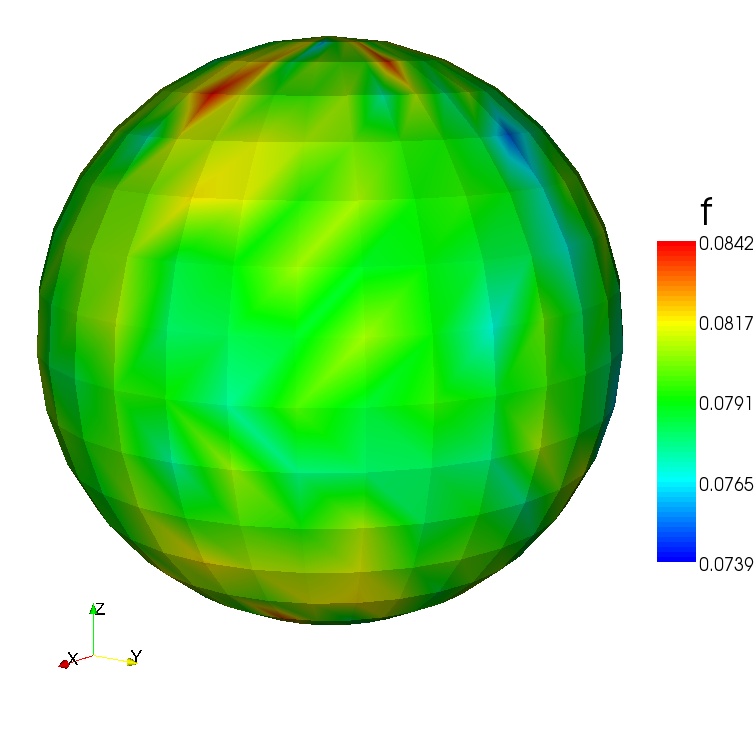}
\includegraphics[width=0.3\textwidth]{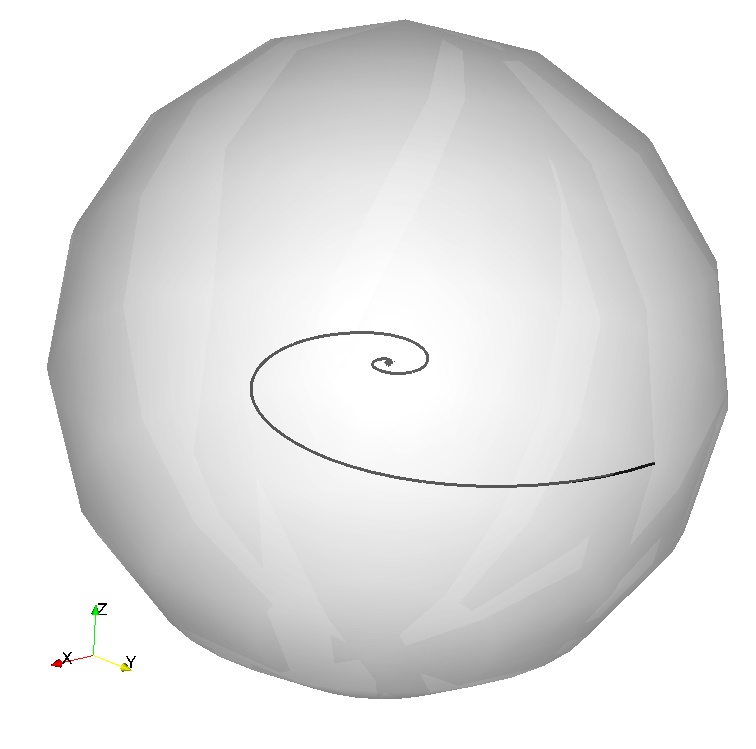}
\includegraphics[width=0.3\textwidth]{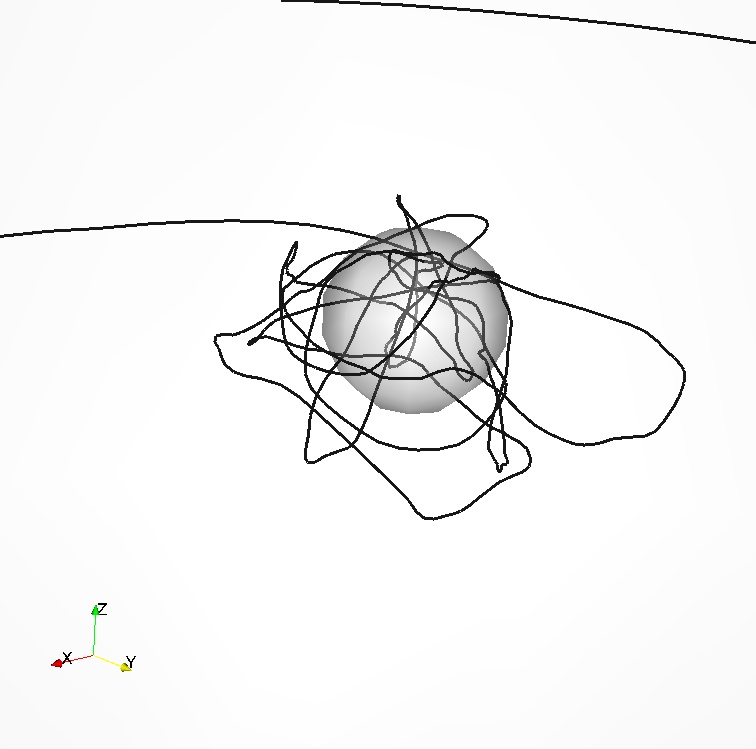}
\caption{Time averaged probability density of $u$ (left), $\mathbb{E}(u(\cdot))$ (middle) and a zoom at $\mathbb{E}(u(\cdot))$ with a sphere with radius $0.01$ (right), $D=1$.}
\label{fig:d1_aver}
\end{figure}

The evolution of the probability density for $D=10$, $T=60$ is shown in Figure~\ref{fig:d10_dens}. Similarly as in the previous experiment the probability density function diffuses and becomes uniform for large time. Some advection in the direction of the initial velocity can still be observed, however, the overall process has a predominantly diffuse character. We observe that the overall evolution damped due to the effects of the random forcing term, see Theorem \ref{numscheme1} and Figure~\ref{fig:conv}. In Figure~\ref{fig:d10_dens} we display the time averaged probability density function, the trajectory $\mathbb{E}(u(t))$ and a zoom at $\mathbb{E}(u(t))$ near the center. 

\begin{figure}[htp]
\center
\includegraphics[width=0.23\textwidth]{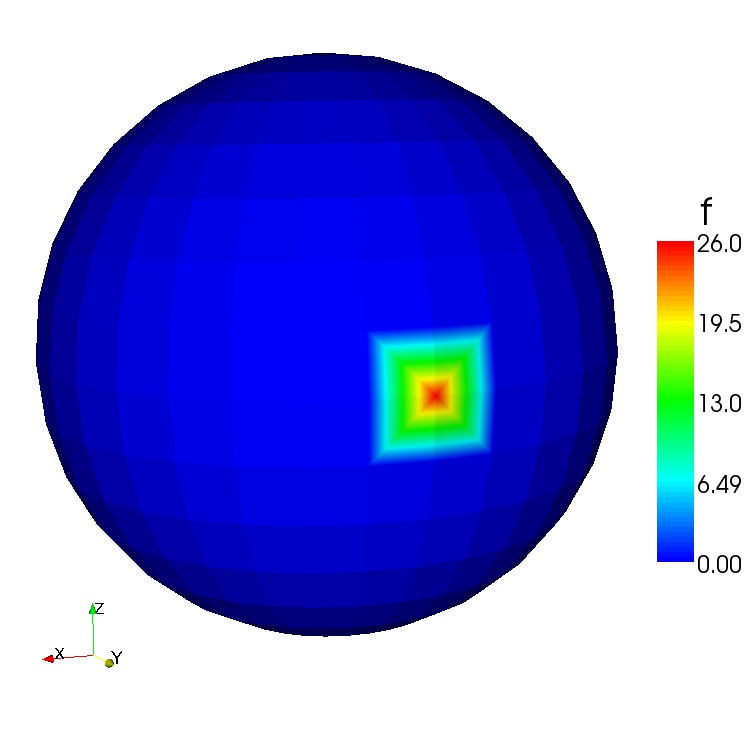}
\includegraphics[width=0.23\textwidth]{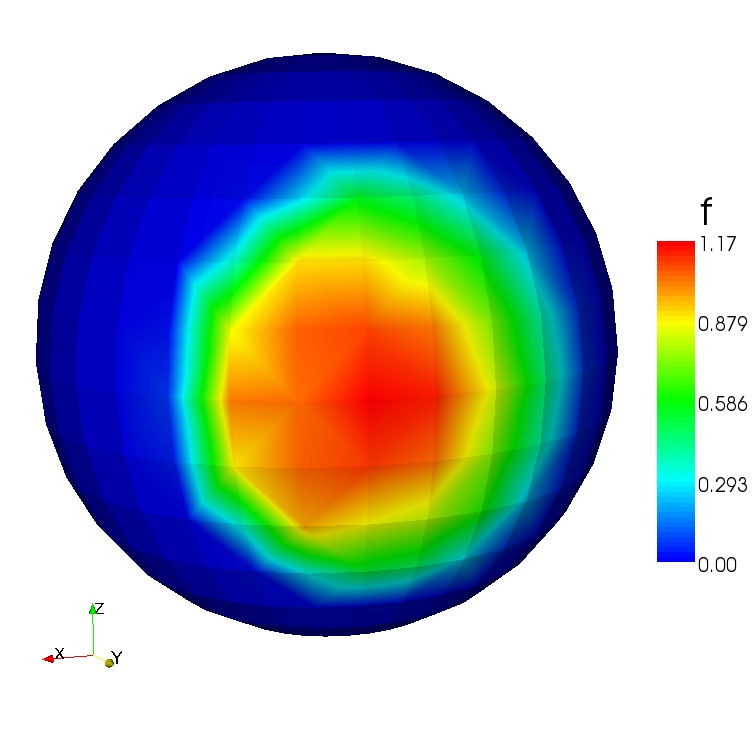}
\includegraphics[width=0.23\textwidth]{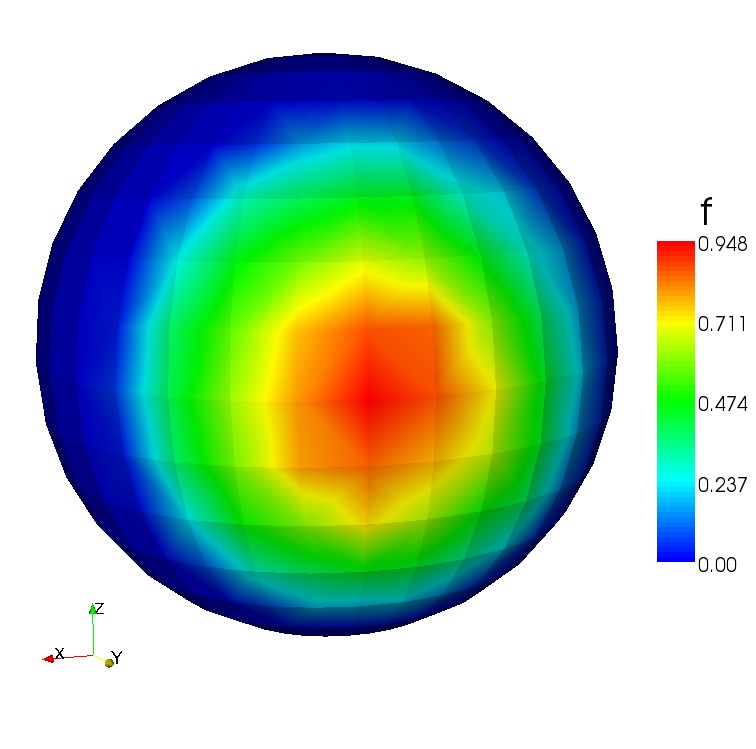}
\includegraphics[width=0.23\textwidth]{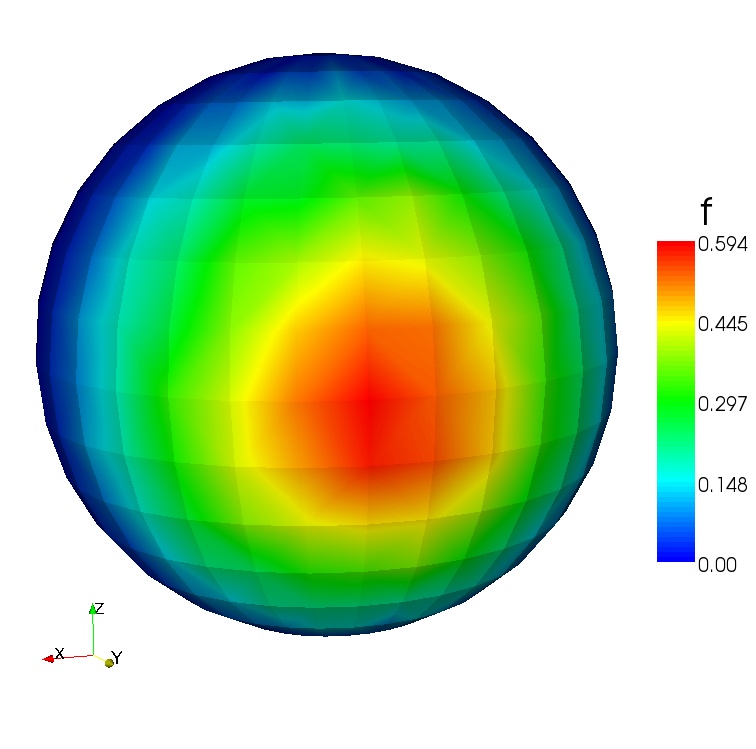}
\includegraphics[width=0.23\textwidth]{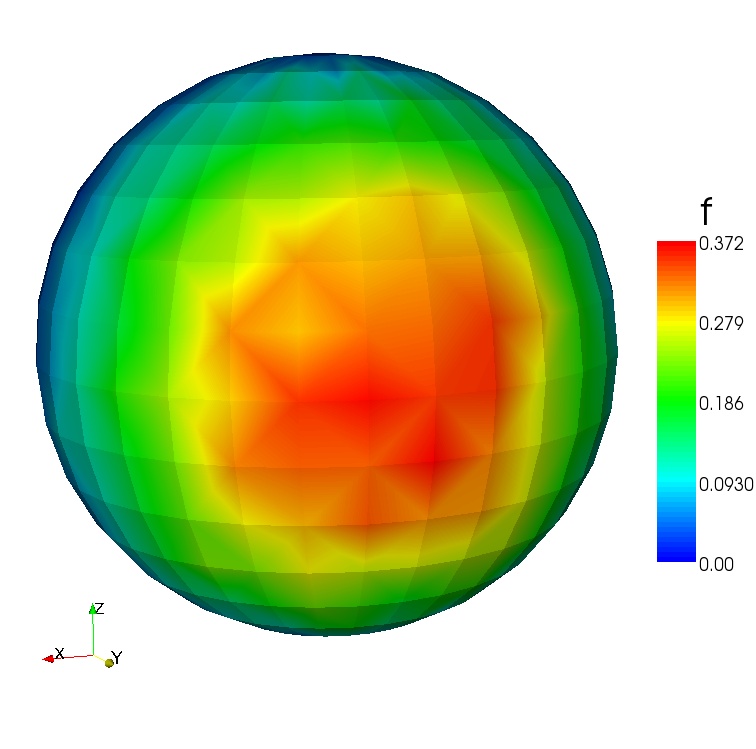}
\includegraphics[width=0.23\textwidth]{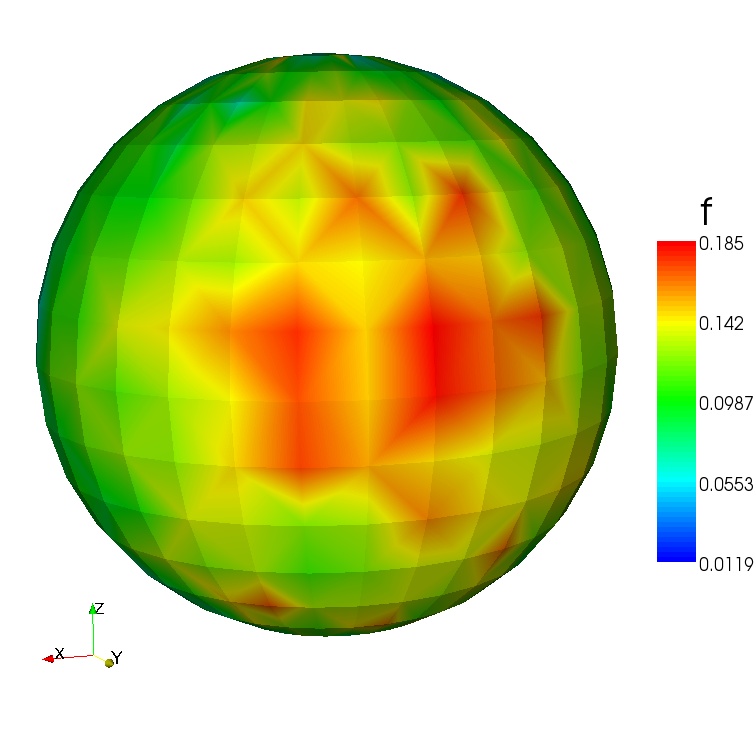}
\includegraphics[width=0.23\textwidth]{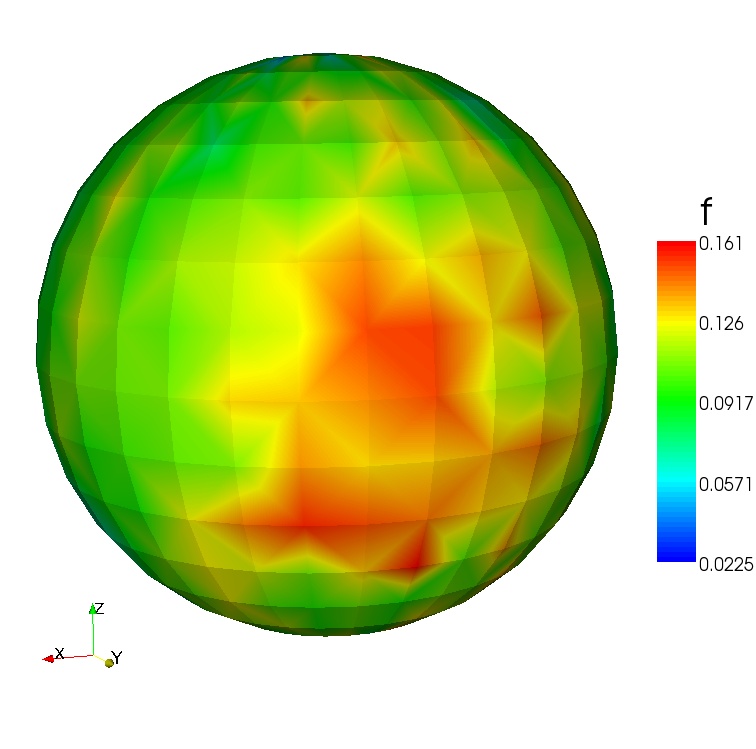}
\includegraphics[width=0.23\textwidth]{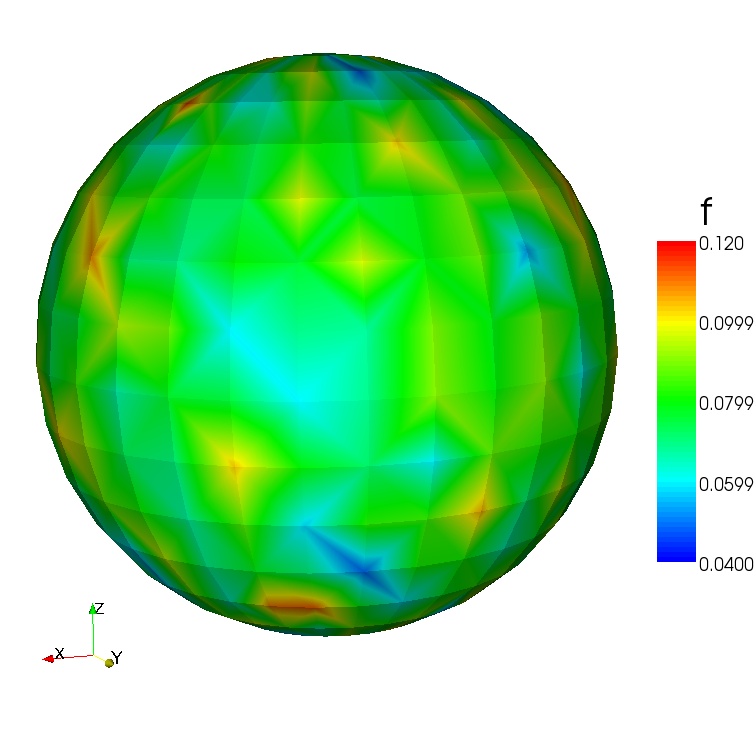}
\caption{Approximate probability density of $u$ for $D=10$ at $t=0,0.9,1.2,2,3.1,8,10,60$.}
\label{fig:d10_dens}
\end{figure}

\begin{figure}[htp]
\center
\includegraphics[width=0.3\textwidth]{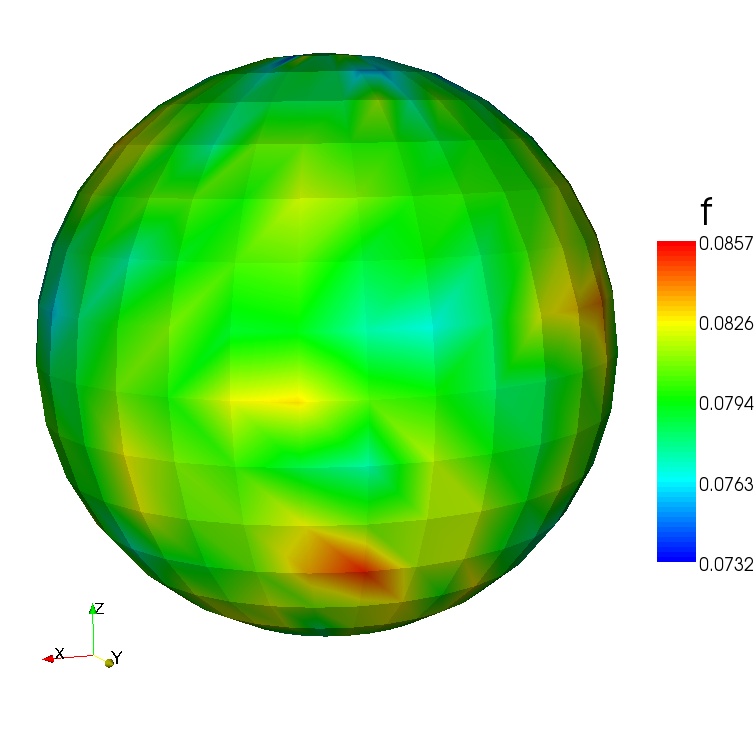}
\includegraphics[width=0.3\textwidth]{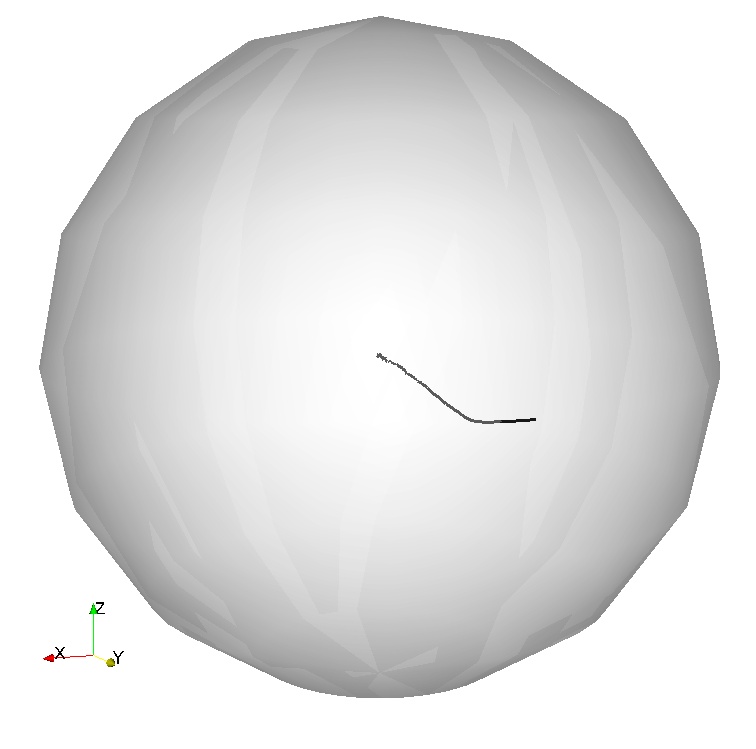}
\includegraphics[width=0.3\textwidth]{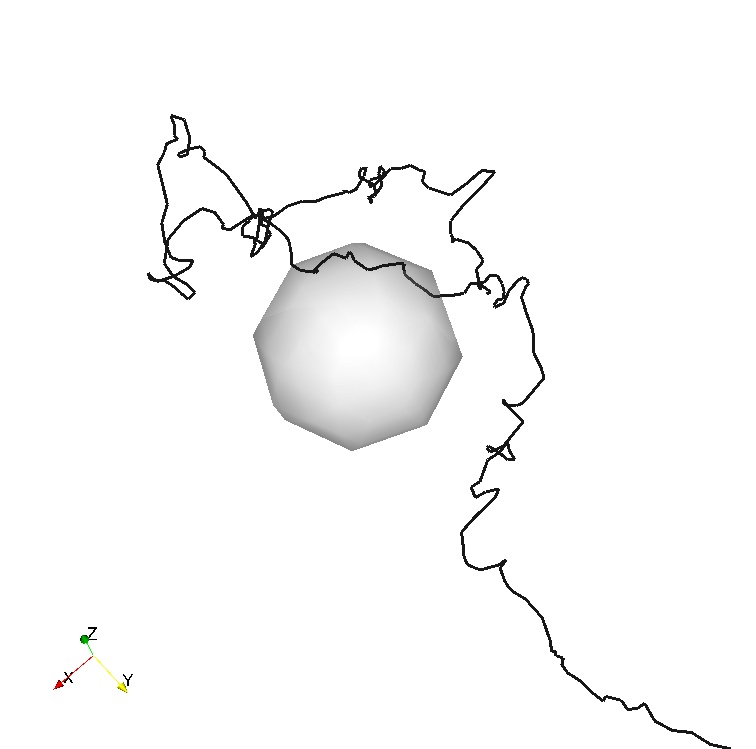}
\caption{Time averaged probability density of $u$ (left), $\mathbb{E}(u(\cdot))$ (middle) and a zoom at $\mathbb{E}(u(\cdot))$ with a sphere with radius $0.01$ (right), $D=10$.}
\label{fig:d10_aver}
\end{figure}

\begin{figure}[htp]
\center
\includegraphics[width=0.3\textwidth]{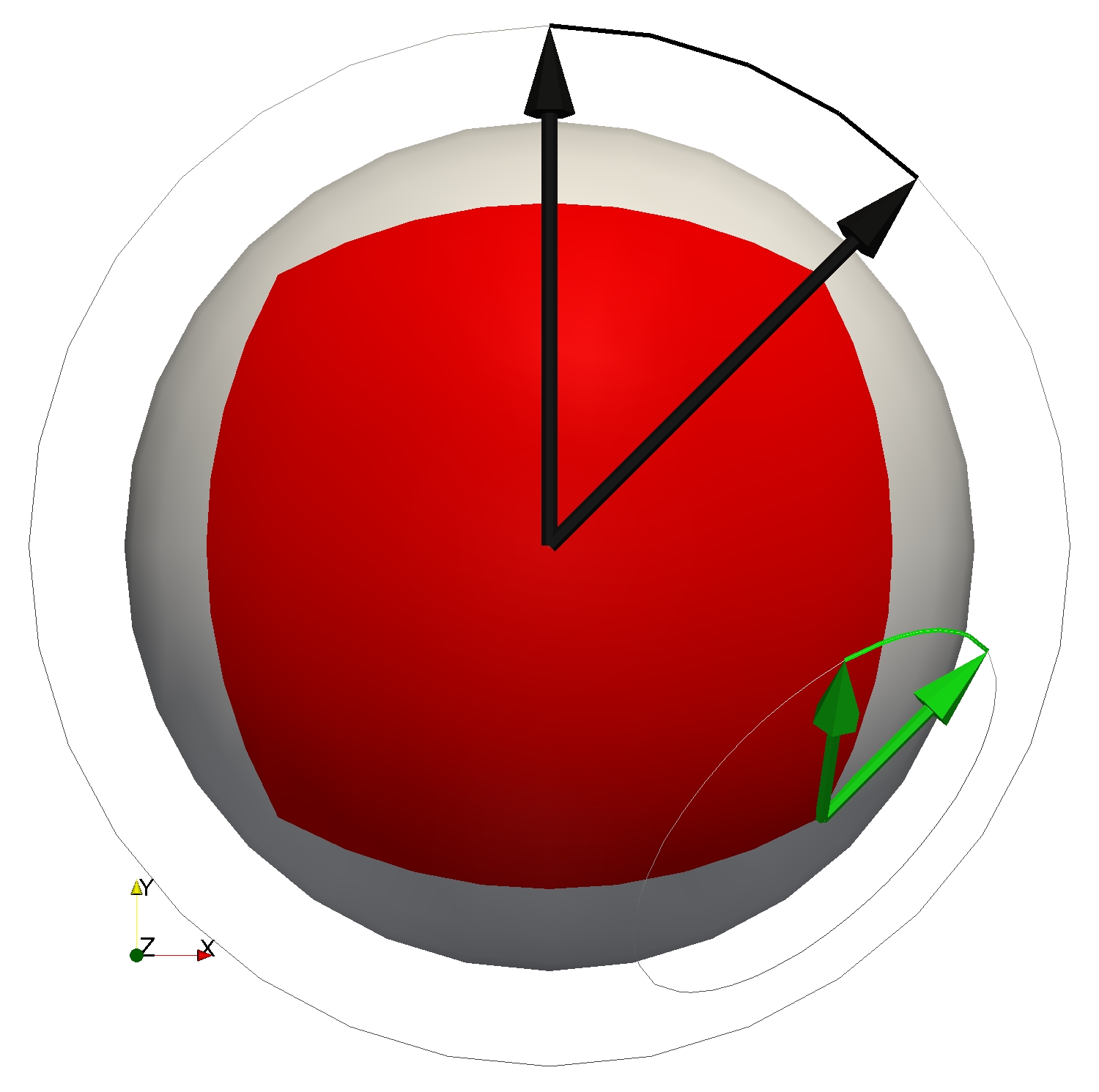}
\caption{The partition of the submanifold $M_1$ of $T\mathbb{S}^2$: $\omega^S_i$ in red, a segment $\gamma^j_i$ in black, the green arc indicates the elements of $M^j_i$ starting from a point in the down-right corner of $\omega^S_i$.}
\label{fig:part_bund}
\end{figure}

Figure~\ref{fig:d100_d0.01} contains the computed trajectories of $\mathbb{E}(u(t))$ for $D=0.1$ and $D=100$.  The respective probability densities asymptotically converge towards the uniform distribution for large time.

\begin{figure}[htp]
\center
\includegraphics[width=0.4\textwidth]{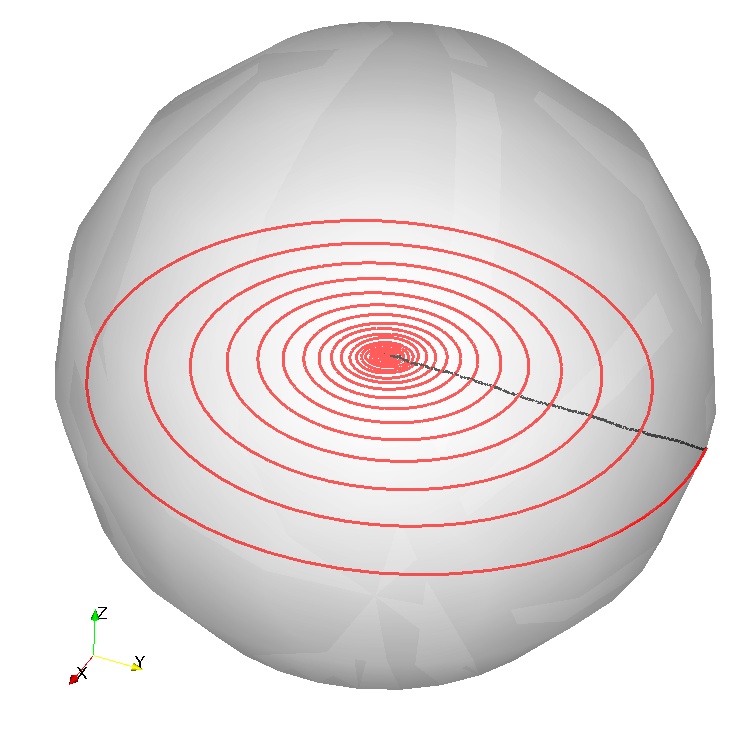}
\includegraphics[width=0.4\textwidth]{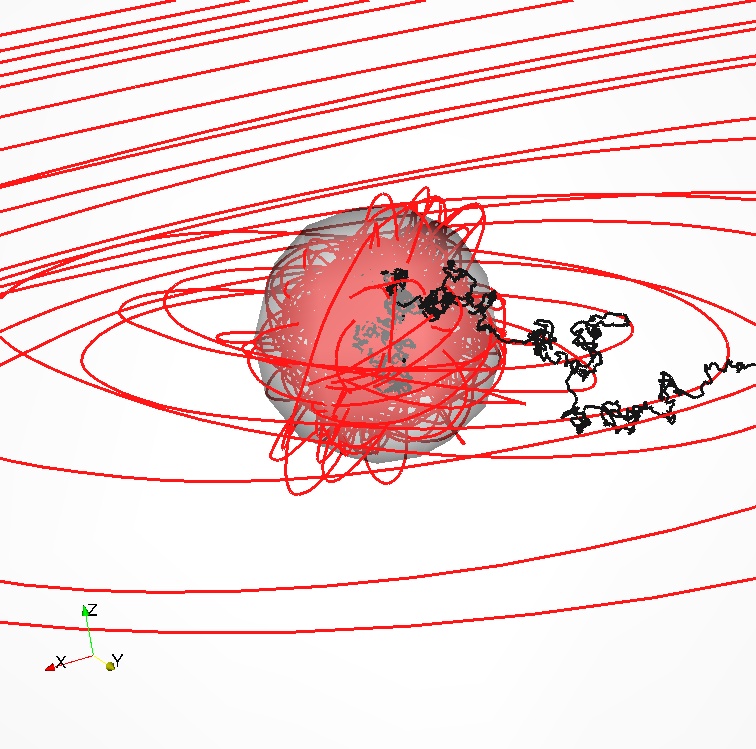}
\caption{$\mathbb{E}(u(\cdot))$ (left) and zoom near the center with s sphere with radius $0.01$ (right) for $D=100$ (black line), $D=0.1$ (red line).}
\label{fig:d100_d0.01}
\end{figure}

In Figure~\ref{fig:conv} we show the graphs of the time evolution of the approximate error $\mathcal{E}_{\max}^n: t^n \rightarrow\max_{\mathbf{x}\in S^2}|f^n(\mathbf{x}) - {f}^S|$ for $D=0.01, 0.1, 1, 10, 100$. The quantity $\mathcal{E}_{\max}^n$ serves as an approximation of the distance from the uniform probability distribution in the $L_{\infty}$ norm. Note that the oscillations in the error graphs are due to the approximation of the probability density. The numerical experiments provide evidence that the  probability densities for all $D$ converges towards the uniform probability density ${f}^{S}$ for $t\rightarrow\infty$. The probability density evolutions for decreasing values of $D$ have an increasingly ``advective'' character and the evolutions for increasing values have an increasingly ``diffusive'' character. It is also interesting to note, that the convergence towards the uniform distribution becomes slower for both increasing and decreasing values of $D$.

\begin{figure}[htp]
\center
\includegraphics[width=0.6\textwidth]{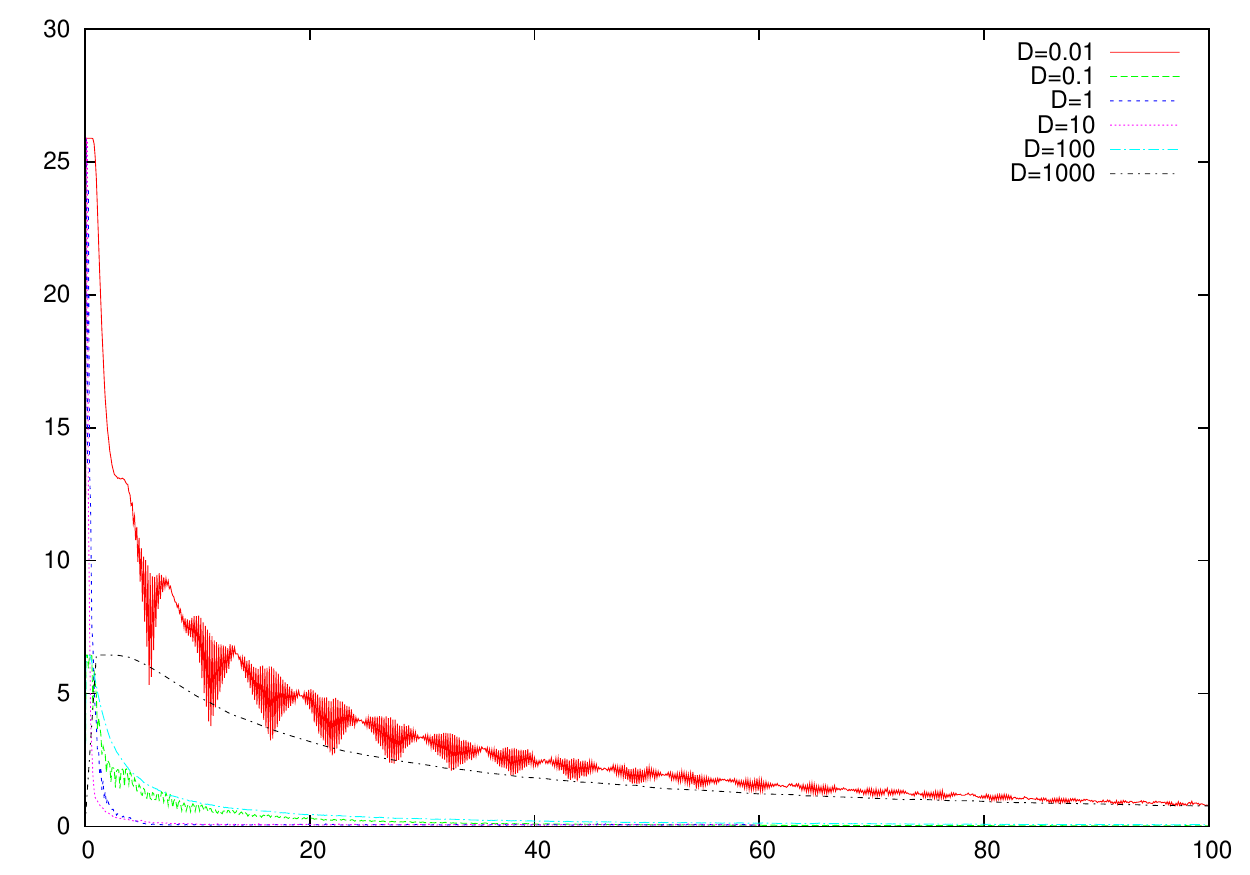}
\caption{Convergence of the probability distribution of $u$ towards a uniform distribution for different values of the coefficient $D$.}
\label{fig:conv}
\end{figure}

In the last experiment we study the long time behavior of the pair $(u,\dot{u})$ for $D=1$, $N=20000$. Towards this end, we introduce a partition of the manifold $M_1$ defined in \eqref{thesetmr}. First, we consider a partition of the unit sphere into segments $\omega^S_i$, $i=1,\dots,6$ associated with the points $x^S_i=(\pm1,0,0),(0,\pm1,0),(0,0,\pm1)$ in such a way that $x\in\mathbb{S}^2$ belongs to $\omega^S_i$ if and only if $|x-x^S_i|=\min_{1\le j\le 6}\,|x-x^S_j|$. Next, we denote by $T_i$ the tangent planes to points $x^S_i$. Fixing an $i\in\{1,\dots,6\}$, the orthogonal projections of vectors $\{x^S_1,\dots,x^S_6\}$ onto the tangent plane $T_i$ delimit $4$ sectors on $T_i$. We subsequently halve each sector obtaining thus $8$ equi-angular sectors $\gamma^1_i,\dots,\gamma^8_i$ in $T_i$. Now we introduce the following partition of $M_1$ into $6\times 8$ segments (see Figure~\ref{fig:part_bund}): a point $(p,\xi)\in T\mathbb{S}^2$ belongs to $M^j_i$ if $p\in\omega^S_i$ and the orthogonal projection of $\xi$ onto the tangent plane $T_i$ belongs to the sector $\gamma^j_i$. It can be verified by symmetries of this partition that the normalized surface volume of each $M^j_i$ is equal to $1/48$. For $n=60000$ (i.e., at time $t^n=60$) we have for $i=1,\dots,6$, $j=1,\dots,8$ $\#\{l|U^n_l \in \omega^S_i\}\in (3380,3260) \approx N/6 = 3333$ and $\#\{l|(U^n_l, {V^n_l})\in M^j_i\}\in (386,455) \approx N/6/8 = 417$, see Figure~\ref{fig:dens_ut} left and  Figure~\ref{fig:dens_ut} right, respectively. The numerical experiments indicate that the point-wise probability measure for $(u,\dot{u})$ converges to the invariant measure $\overline{\nu}$. The (rescaled) approximate $L_\infty$ error $\mathcal{E}_{\max}$ for $(u,\dot{u})$ has similar evolution as  the approximate  $L_\infty$ error for $u$. Moreover, it seems that the convergence of the error in time is exponential, see Figure~\ref{fig:linftybundle}. 

\begin{figure}[htp]
\center
\includegraphics[width=0.3\textwidth]{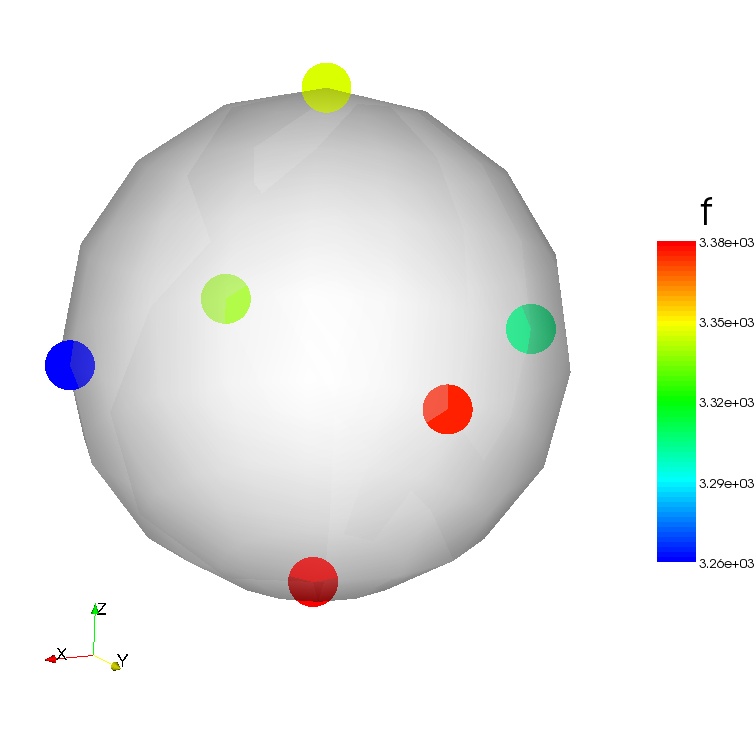}
\includegraphics[width=0.3\textwidth]{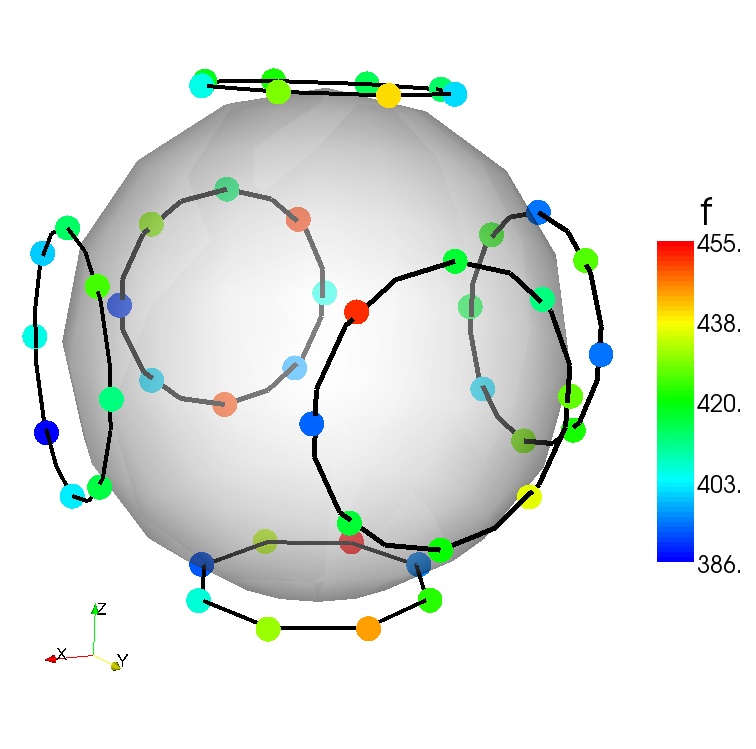}
\includegraphics[width=0.3\textwidth]{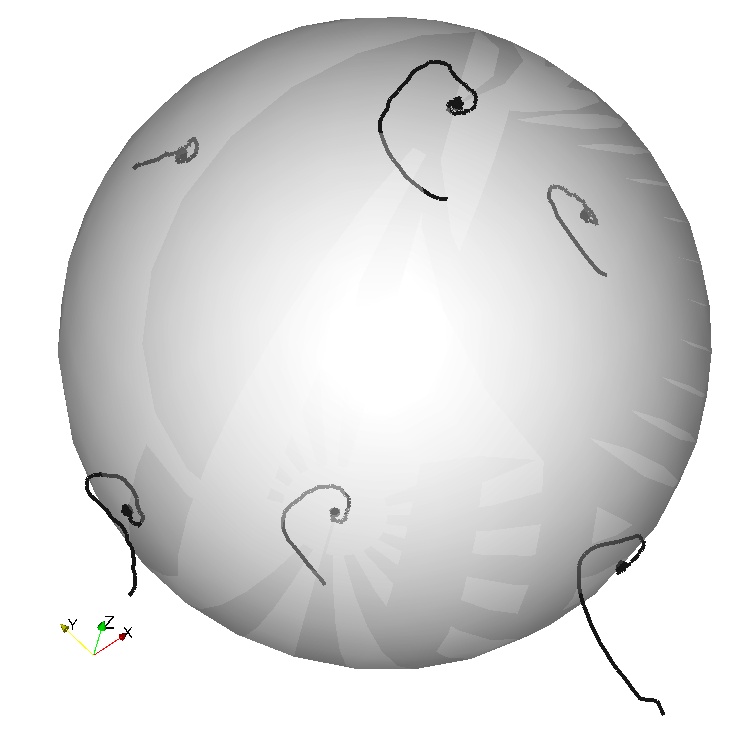}
\caption{Probability distribution of $(u,\dot{u})$ at $T=60$ (left and middle) and the evolution of $t\rightarrow \int_{\omega^S_i}\mathbb{E}(\dot{u}(t))$, $i=1,\dots,6$ (right).}
\label{fig:dens_ut}
\end{figure}

\begin{figure}[htp]
\center
\includegraphics[width=0.6\textwidth]{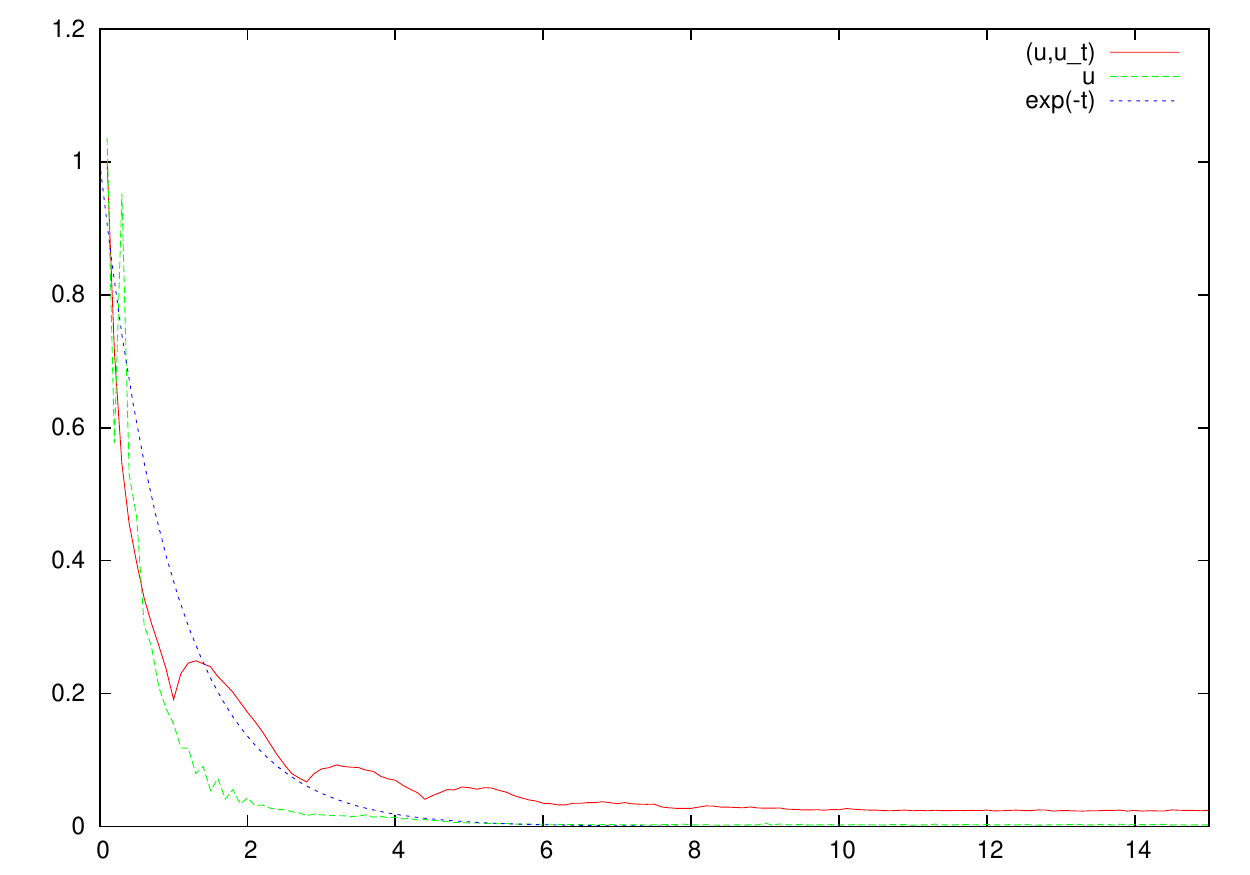}
\caption{Approximate $L_\infty$ convergence of the respective probability distributions for $(u,\dot{u})$ and $u$ towards a uniform distribution.}
\label{fig:linftybundle}
\end{figure}

\section{Invariant measures on $M_r$, $r>0$}

It is known that equations on manifolds with non-degenerate diffusions have a unique invariant probability law, that this invariant measure is absolutely continuous with respect to the surface measure and the density is $C^\infty$-smooth and strictly positive, see e.g. \cite{arnkli} or \cite[Proposition 4.5]{ikwa}. Unfortunately, the equation \eqref{gsde} on $M_r$ has a degenerate diffusion - there is just one vector field $g$ in the diffusion but $M_r$ is a $3$-dimensional manifold. In other words, there is not enough noise in the equation in order the above cited results on the nice ergodic behaviour could be applied in our case. We must therefore proceed in another way to confirm the conjectures of  Section \ref{numsoll}.

\begin{convention}\label{restrmr}
In the present section, we restrict the operators $(P_t)$ and $(P^*_t)$ to the invariant space $M_r$ where $r>0$ is fixed. More precisely, $P_t$ is understood as an endomorphism on $B_b(M_r)$ and $P^*_t$ is an endomorphism on the space of probability measures on $\mathscr B(M_r)$, cf. Theorem \ref{semigrprop}. Also $M_r$ is understood as a submanifold in $\Bbb R^6$.
\end{convention}

\begin{definition} We denote by $\lambda_r$ the normalized surface (Riemannian) measure on $M_r$.
\end{definition}

\subsection{Uniqueness} We are going to prove, using the geometric version of the H\"ormander theorem \ref{hormanthm} that $\lambda_r$ is the unique invariant measure on $M_r$. But let us first, before we proceed with the study of the qualitative properties of the adjoint Markov semigroup $(P^*_t)$, establish some furhter geometric properties of the drift and the diffusion vector fields $f$ and $g$ defined in \eqref{defvf}.

\begin{lemma}\label{fgprop} $M_r$ is a connected $3$-dimensional submanifold in $\Bbb R^6$ and the vector fields $f$ and $g$ on $M_r$ satisfy 
$$
[g,f]=\left(\begin{array}{c}u\times v\\0\end{array}\right),\qquad [f,[g,f]]=r^2g,\qquad [g,[g,f]]=-f,\qquad\operatorname{div}\,f=\operatorname{div}\,g=\operatorname{div}\,[g,f]=0
$$
where $[\cdot,\cdot]$ is the Jacobi bracket.
\end{lemma}

\begin{proof} Obviously, any $(p,\xi_1)$ and $(p,\xi_2)$ in $M_r$ can be connected by a rotation curve in the circle $\{(p,\xi):\xi\perp p,\,|\xi|=r\}$ and if $|p|=|q|=1$ and $\gamma:[a,b]\to\Bbb S^2$ is a curve connecting $p$ and $q$ with $|\dot\gamma|=r$ then $\Gamma=(\gamma,\dot\gamma)$ is a curve connecting $(p,\dot\gamma(a))$ and $(q,\dot\gamma(b))$ in $M_r$. Altogether, any two points in $M_r$ can be connected by an at most two times broken curve.

Observe that $f$, $g$ and $[g,f]$ are orthogonal tangent vector fields on $M_r$. If we define $E_1=f/(r^2+r^4)^\frac 12$, $E_2=g/r$, $E_3=[G,F]/r$ then $\{E_1,E_2,E_3\}$ is an othonormal frame on $M_r$ and
$$
\operatorname{div}Y=\sum_{j=1}^3\langle d_{E_j}Y,E_j\rangle_{\mathbb{R}^6}=0,\qquad Y\in\{f,g,[f,g]\}
$$
where $d_XY(p)=\lim_{t\to 0}t^{-1}[Y(p+tX)-Y(p)]$.
\end{proof}

\begin{definition} Let $S^1,\dots,S^m$ be vector fields on a manifold $M$. Denote by $\mathscr (S^1,\dots,S^m)$ the least algebra for the Jacobi bracket $[X,Y]=XY-YX$ that contains $\{S^1,\dots,S^m\}$ and denote
$$
\mathscr L(S^1,\dots,S^m)(p)=\{S_p:S\in\mathscr L(S^1,\dots,S^m)\}\subseteq T_pM,\qquad p\in M.
$$
\end{definition}

\begin{corollary}\label{horm} $\mathscr L(f,g)(z)=T_zM_r$ holds for every $z\in M_r$.
\end{corollary}

The following result is known\footnote{See e.g. (4.58) on p. 292 in \cite{ikwa}.} but we can give its straight analytic proof in few lines now.

\begin{proposition}\label{prop-charinv} A  probability measure $\nu$ on $\mathscr B(M_r)$ is invariant if and only if
\begin{equation}\label{baseqinvm}
\int_{M_r}\mathcal Ah\,d\nu=0\quad\textrm{for every}\quad h\in C^2(M_r)
\end{equation}
where the operator $\mathcal A$ was defined in \eqref{defopera}.
\end{proposition}
\begin{proof} This is an immediate consequence of the $C_0$-semigroup property of $(P_t)$ on $C(M_r)$, the invariance of $C^2(M_r)$ under $(P_t)$, the fact that  $P_t\circ\mathcal A=\mathcal A\circ P_t$ on $C^2(M_r)$ for every $t\ge 0$ and density of $C^2(M_r)$ in $B_b(M_r)$ as all proved in Theorem \ref{semigrprop}.
\end{proof}

\begin{proposition}\label{prop-thdiv} Let $R\in C^2(M_r)$. Then the measure $d\nu=R\,d\lambda_r$ satisfies \eqref{baseqinvm} iff $R$ is constant on $M_r$.\end{proposition}

\begin{proof}
Using the standard formulae
$$
\int_{M_r}Yh\,d\lambda_r
=-\int_{M_r}h\operatorname{div}Y\,d\lambda_r,\qquad\operatorname{div}(hY)=Y(h)+h\operatorname{div}Y
$$
that hold for any smooth vector field $Y$ on $M_r$ and any smooth function $h$ on $M_r$, applying Lemma \ref{fgprop} and Proposition \ref{prop-charinv} and using the fact that $C^2(M_r)$ is dense in $L^1(M_r,\lambda_r)$, we get that $\nu$ satisfies \eqref{baseqinvm} iff 
the identity
\begin{equation}\label{dive}
fR=\frac 12g(gR)
\end{equation}
holds on $M_r$. But
$$
\int_{M_r}R(fR-\frac 12g(gR))\,d\lambda_r=\frac 12\int_{M_r}|gR|^2\,d\lambda_r
$$
as $f$ and $g$ are divergence-free, so we conclude that \eqref{dive} holds iff $gR=fR=0$. If $R$ is constant, this equality surely holds. For the converse implication, by definition of the Lie bracket, $[g,f]R=0$ holds. Since $f_z$, $g_z$ and $[g,f]_z$ span $T_zM_r$ for every $z\in M_r$ by Lemma \ref{fgprop}, we obtain that $R$ is locally constant. But $M_r$ is connected by Lemma \ref{fgprop}, hence $R$ is constant.
\end{proof}

\begin{theorem}\label{prop-horem} $\lambda_r$ is the unique invariant probability measure on $M_r$.
\end{theorem}

\begin{proof}
Let $\nu$ be an inavriant measure. Since \eqref{baseqinvm} holds and the geometric version of the H\"ormander theorem \ref{hormanthm} is applicable due to Corollary \ref{horm}, we conclude that $\nu$ has a smooth density $R$ with respect to $\lambda_r$. But then $R=1$ on $M_r$ by Proposition \ref{prop-thdiv}.
\end{proof}

\section{The transition probabilities on $M_r$, $r>0$}

In this section, we continue the study of the Markov semigroup $(P_t)$ and its adjoint semigroup $(P^*_t)$ restricted to $M_r$ as set forth in Convention \ref{restrmr}, with $r>0$ fixed. We are going to show that the transition probabilities $p_{t,x}$ restricted to $\mathscr B(M_r)$ for $x\in M_r$ are absolutely continuous with respect to the normalized surface measure $\lambda_r$ on $M_r$ for every $(t,x)\in(0,\infty)\times M_r$ and that the density $p(t,x,\cdot)$ satisfies $p\in C^\infty((0,\infty)\times M_r\times M_r)$. The density $p(t,x,\cdot)$ should be denoted by $p_r(t,x,\cdot)$ to indicate the dependence on $r>0$ but we will not use this notation since $r$ is fixed in this section and we will not use the densities elsewhere in this paper.

An expert could be simply advised to apply the abstract results based on the geometric H\"ormander theorem in \cite[Theorem 3]{ich1} but we prefer to guide the reader through, to explain the  actual structure of the problem better.

For, let us define the adjoint operator
\begin{equation}
\mathcal A^*h=-f(h)+\frac 12g(g(h)),\qquad h\in C^2(M_r)
\end{equation}
to the operator $\mathcal A$ defined in \eqref{defopera}. Indeed, by Lemma \ref{fgprop}, 
\begin{equation}\label{duality}
\int_{M_r}(\mathcal Ah_1)h_2\,d\lambda_r=\int_{M_r}h_1\mathcal A^*h_2\,d\lambda_r,\qquad\forall h_1,h_2\in C^2(M_r)
\end{equation}
as $f$ and $g$ are divergence-free on $M_r$.

\begin{theorem}\label{smoothdens} The transition probabilities $p_{t,x}$ are absolutely continuous with respect to the normalized surface measure $\lambda_r$ on $M_r$ for every $(t,x)\in(0,\infty)\times M_r$ and the density $p(t,x,\cdot)$ satisfies $p\in C^\infty((0,\infty)\times M_r\times M_r)$.
\end{theorem}

\begin{proof}
Consider the Riemannian manifold $N=(0,\infty)\times M_r\times M_r$ and define the Radon measure
$$
\Gamma(A)=\int_0^\infty\int_{M_r}\int_{M_r}\mathbf 1_A(t,x,z)\,dp_{t,x}(z)\,d\lambda_r(x)\,dt=\Bbb E\int_0^\infty\int_{M_r} 1_A(t,x,z^x(t))\,dt\,d\lambda_r(x),\,A\in\mathscr B(N).
$$
Every function $h\in C^\infty(N)$ has variables $(t,x,z)$ and we are going to indicate by $\mathcal A_z$ that the operator $\mathcal A$ is applied on the variable $z$ and by $\mathcal A^*_x$ that the operator $\mathcal A^*$ is applied on the variable $x$ of the function $h(t,x,z)$.

By the It\^o formula,
\begin{equation}\label{somothito1}
\int_0^\infty\int_{M_r}\left(\frac{\partial H}{\partial t}+\mathcal AH\right)\,dp_{t,x}\,dt=0\qquad\forall H\in C^\infty_{comp}((0,\infty)\times M_r)
\end{equation}
holds for every $x\in M_r$ hence
\begin{equation}\label{somothito2}
\int_N\left(\frac{\partial h}{\partial t}+\mathcal A_zh\right)\,d\Gamma=0\qquad\forall h\in C^\infty_{comp}(N).
\end{equation}
Let $h_1\in C^\infty_{comp}(0,\infty)$, $h_2,h_3\in C^\infty(M_r)$ and define $H(t,x)=h_1(t)h_2(x)$, $h(t,x,z)=h_1(t)h_2(x)h_3(z)$ and $v(t,x)=P_th_3(x)$. Then
$$
\int_N\left(\frac{\partial h}{\partial t}+\mathcal A^*_xh\right)\,d\Gamma=\int_0^\infty\int_{M_r}\left(\frac{\partial H}{\partial t}+\mathcal A^*H\right)v\,d\lambda_r\,dt=\int_0^\infty\int_{M_r}H\left(-\frac{\partial v}{\partial t}+\mathcal Av\right)\,d\lambda_r\,dt=0
$$
by \eqref{kolm2} and the duality \eqref{duality}. In fact, 
\begin{equation}\label{somothito3}
\int_N\left(\frac{\partial h}{\partial t}+\mathcal A^*_xh\right)\,d\Gamma=0,\qquad\forall h\in C^\infty_{comp}(N)
\end{equation}
by a density argument as shown in Proposition \ref{density}.

Altogether we have obtained that 
$$
\int_N\left(2\frac{\partial h}{\partial t}+\mathcal A^*_xh+\mathcal A_zh\right)\,d\Gamma=0,\qquad\forall h\in C^\infty_{comp}(N).
$$
In order to apply the geometric H\"ormander theorem \ref{hormanthm}, we define the vector fields
$$
Y(t,x,z)=\left(\begin{array}{c}2\\-f(x)\\f(z)\end{array}\right),\qquad X^1 (t,x,z)=\left(\begin{array}{c}0\\g(x)\\0\end{array}\right),\qquad X^2(t,x,z)=\left(\begin{array}{c}0\\0\\g(z)\end{array}\right)
$$
where the vector field $Y$ corresponds to the operator $h\mapsto 2\frac{\partial h}{\partial t}-f_x(h)+f_y(h)$, the vector field $X^1$ to the operator $h\mapsto g_x(h)$ and the vector field $X^2$ to the operator $h\mapsto g_z(h)$. Defining also $h=[g,f]$ on $M_r$, we get by Lemma \ref{fgprop} that
$$
[Y,X^1]=\left(\begin{array}{c}0\\h(x)\\0\end{array}\right),\quad [Y,X^2]=-\left(\begin{array}{c}0\\0\\h(z)\end{array}\right),\quad [X^1,X^2]=0,
$$
$$
[Y,[Y,X^1]=-r^2X^1,\quad[Y,[Y,X^2]]=-r^2X^2,\quad[X^1,[Y,X^1]=-\left(\begin{array}{c}0\\f(x)\\0\end{array}\right),\quad [X^1,[Y,X^2]]=0,
$$
$$
[X^2,[Y,X^1]=0,\quad[X^2,[Y,X^2]]=\left(\begin{array}{c}0\\0\\f(z)\end{array}\right),\quad[[Y,X^1],[Y,X^2]]=0.
$$
At this stage we see that 
$$
\mathscr L(Y,X^1,X^2)(t,x,z)=\Bbb R\times T_xM_r\times T_zM_r=T_{(t,x,z)}N,\qquad\forall (t,x,z)\in N
$$
so the geometric H\"ormander theorem \ref{hormanthm} is applicable and $\Gamma$ has a smooth density $p\in C^\infty(N)$ with respect to $dt\otimes\lambda_r\otimes\lambda_r$.

Let $\varphi\in C(M_r)$. Then, by the standard measure theoretical properties of integrals,
\begin{equation}\label{fyty}
P_t\varphi(x)=\int_{M_r}\varphi(z)p(t,x,z)\,d\lambda_r(z)
\end{equation}
holds for $dt\otimes\lambda_r$-almost every $(t,x)$. But since both sides are continuous in $(t,x)$ (the right hand side by Theorem \ref{semigrprop}), the identity \eqref{fyty} holds for every $(t,x)\in(0,\infty)\times M_r$. By standard procedure, we extend \eqref{fyty} to hold for every $\varphi\in B_b(M_r)$ and every $(t,x)\in(0,\infty)\times M_r$.
\end{proof}

The following result recasts Corollary \ref{rotations} in terms of the transition densities.

\begin{corollary}\label{rotations_dens} Let $Q$ be a $3\times 3$-unitary matrix. Denote by $\widetilde Q=\operatorname{diag}\,[Q,Q]$. Then 
$$
p(t,x,y)=p(t,\widetilde Qx,\widetilde Qy)
$$
holds for every $(t,x,y)\in(0,\infty)\times M_r\times M_r$.
\end{corollary}

\begin{proof}
We just realize that $\widetilde Q$ is a measure preserving diffeomorphism on $M_r$ (as a restriction of an isometry on $\Bbb R^6$) and then we apply Corollary \ref{rotations}.
\end{proof}

\section{Controlability in $M_r$, $r>0$}

In this section, we are going to examine the supports of the probability measures $p_{t,x}$ on $\mathscr B(M_r)$ for $x\in M_r$. Again, in this section, the Markov semigroup $(P_t)$ and its adjoint semigroup $(P^*_t)$ are restricted to $M_r$ as in Convention \ref{restrmr}, with $r>0$ fixed.

\begin{theorem}\label{supi} Let $t\ge 2\pi/r$. Then $\operatorname{supp}\,p_{t,x}=M_r$ holds for every $x\in M_r$.
\end{theorem}

\subsection{General support result}

Let $x\in M_r$ and denote by $V^{x,a}$ the solutions of the ordinary differential equation 
\begin{equation}\label{adithree}
X^\prime=f(X)+a(t)g(X),\qquad X(0)=x
\end{equation}
on $M_r$ where $a\in L^1_{loc}([0,\infty))$ and $f$ and $g$ are defined in \eqref{defvf}.

\begin{remark} It can be checked analogously as in the proof of Proposition \ref{gext} that the solutions $V^{x,a}$ take values in $M_r$ and are therefore global.
\end{remark}

The next lemma tells us that, to describe the support of the probabilities $p_{t,x}$ for $x\in M_r$, it is sufficient and necessary to study solutions of the ordinary differential equation \eqref{adithree}.

\begin{lemma}\label{lem93} Let $t>0$ and $x\in M_r$. Then 
\begin{equation}\label{aditfive}
\operatorname{supp}\,p_{t,x}=\overline{\{V^{x,a}(t):a\in L^1(0,t)\}}^{M_r}.
\end{equation}
\end{lemma}

\begin{proof}
Let $\tilde f$, $\tilde g$ be smooth compactly supported vector fields on $\Bbb R^6$ and denote by $\mu$ the law of the solution of the equation
\begin{equation}\label{adione}
dX=\tilde f(X)+\tilde g(X)\circ dW,\qquad X(0)=x
\end{equation}
on $\mathscr B(C([0,t];\Bbb R^6))$. Let also $a\in L^1(0,t)$ and denote by $v^a$ the solution of
\begin{equation}\label{aditwo}
X^\prime=\tilde f(X)+a(t)\tilde g(X),\qquad X(0)=x.
\end{equation}
Then, according to the Support theorem of Stroock and Varadhan \cite{strvar} (see also \cite{aks}, \cite{bengra}, \cite{bagl}, \cite{gyopro}, \cite{macke} for generalizations or shorter proofs),
$$
\operatorname{supp}\mu=\overline{\{v^a:a\in L^1(0,t)\}}
$$
where the closure and the support are taken in $C([0,t];\Bbb R^6)$. Since $v^{a_n}\to v^a$ uniformly on $[0,t]$ if $a_n\to a$ in $L^1(0,t)$ and $A$ is a dense subset in $L^1(0,t)$, it also holds
$$
\operatorname{supp}\mu=\overline{\{v^a:a\in A\}}.
$$
To get back to our problem \eqref{gsde}, let $\tilde f$, $\tilde g$ be smooth compactly supported vector fields on $\Bbb R^6$ such that $\tilde f=f$ and $\tilde g=g$ on the centered ball in $\Bbb R^6$ of the radius $R=|x|$. Then the solution $X$ coincides with $z^x$ being the solution of \eqref{gsde} with $z^x(0)=x$. Also, by uniqueness,  $V^{x,a}=v^a$. Thus we conclude that 
\begin{equation}\label{aditfour}
\operatorname{supp}\,(\operatorname{Law}\,z^x)=\overline{\{V^{x,a}:a\in L^1(0,t)\}}
\end{equation}
where both the support and the closure are taken in $C([0,t];M_r)$ being a closed subset of $C([0,t];\Bbb R^6)$.

Now consider the projection $\pi_t:C([0,t];M_r)\to M_r:\xi\mapsto\xi(t)$. Since $\pi$ is continuous, 
$$
\overline{\pi_t[\operatorname{supp}\,(\operatorname{Law}\,z^x)]}=\operatorname{supp}\,(\operatorname{Law}\,\pi_t(z^x))=\operatorname{supp}\,p_{t,x},
$$
and by continuity of $\pi_t$ and \eqref{aditfour}, we also have
$$
\overline{\pi_t[\operatorname{supp}\,(\operatorname{Law}\,z^x)]}=\overline{\pi_t[\overline{\{V^{x,a}:a\in L^1(0,t)\}}]}=\overline{\pi_t[\{V^{x,a}:a\in L^1(0,t)\}]}=\overline{\{V^{x,a}(t):a\in L^1(0,t)\}}.
$$
\end{proof}

\subsection{The control problem} In view of Lemma \ref{lem93}, it remains to prove that the ordinary differential equation \eqref{adithree} can be controlled to hit every point in $M_r$ after time $2\pi/r$. It turns out that it is necessary to enter deeper to the geometry of the $2D$ sphere.

For consider the equation \eqref{adithree} with a constant control $a\in\Bbb R$
\begin{equation}\label{aditsix}
w^{\prime\prime}=-|w^\prime|^2w+a w\times w^\prime
\end{equation}
and with the initial condition $w(0)=p$, $w^\prime(0)=\xi$ for $x=(p,\xi)\in M_r$. It can be guessed (and consequently checked) from rotational symmetries of \eqref{aditsix} that the unique solution has the form
\begin{equation}\label{wsol}
w^{x,a}(t)=\frac{a}{b}E_1^{x,a}+\frac{r}{b}E_2^{x,a}\cos(bt)+\frac{r}{b}E_3^{x,a}\sin(bt)
\end{equation}
$$
E^{x,a}_1=\frac{a}{b}p+\frac{1}{b}p\times\xi,\quad E^{x,a}_2=\frac{r}{b}p-\frac{a}{rb}p\times\xi,\quad E^{x,a}_3=\frac 1r\xi
$$
where $b=\sqrt{r^2+a^2}$. Since $\{E^{x,a}_1,E^{x,a}_2,E^{x,a}_3\}$ is orthonormal with $\operatorname{det}\,(E^{x,a}_1,E^{x,a}_2,E^{x,a}_3)=1$, we deduce that $w^{x,a}$ is a parametrization of a circle on $\Bbb S^2$ with the derivative of constant length $r$. 

\begin{lemma}
A $C^2$-smooth curve such that $|w|_{\Bbb R^3}=1$ and $|w^\prime|_{\Bbb R^3}=r$ satisfies the equation \eqref{aditsix} for some control $a\in\Bbb R$ iff it parametrizes a non-degenerate circle\footnote{Here ``non-degenerate'' means that the radius of the circle is strictly positive.} on $\Bbb S^2$.
\end{lemma}

Hence, solutions of \eqref{adithree} can be regarded as oriented circles in $\Bbb S^2$. 

\begin{definition} In the sequel, we are going to consider pairs $(K,Y)$ where $K$ is a non-degenerate circle on $\Bbb S^2$ and $Y$ is a vector field on $K$ with $|Y_p|=r$ for every $p\in K$. Such pairs are going to be called oriented circles in $\Bbb S^2$ for simplicity.
\end{definition}

\begin{remark}\label{somepropcirc} Any non-degenerate circle $K$ in $\Bbb S^2$ can be described in a unique way as $K=(v+P)\cap\Bbb S^2$ where $P$ is a two-dimensional subspace in $\Bbb R^3$, $v\in\Bbb R^3$ is perpendicular to $P$ and $|v|<1$. Here the vector $v$ is the center of the circle $K$ and $P$ is the plane of the circle. Obviously, if $\bar v\in\Bbb R^3$ then $K=(\bar v+P)\cap\Bbb S^2$ iff $\bar v-v\in P$. Also 
$$
T_zK=\{p\in P:p\perp z\}=\{p\in P:p\perp (z-v)\},\qquad\forall z\in K.
$$
If we define $\theta=\sqrt{1-|v|^2}$, $\{p_1,p_2\}$ is an orthonormal basis in $P$ and
$$
Y_z=\frac{r}{\theta}\left[-\langle z,p_1\rangle p_2+\langle z,p_2\rangle p_1\right],\qquad z\in K
$$
then $\{Y,-Y\}$ are the only two vector fields on $K$ of length $r$.
\end{remark}

\begin{lemma} Let $x=(p,\xi)\in M_r$ and define the circle on $\Bbb S^2$
$$
K=(p+\operatorname{span}\,\{E^{x,a}_2,E^{x,a}_3\})\cap\Bbb S^2
$$
in the notation of \eqref{wsol} and the vector field on $K$ of length $r$
$$
Y(z)=-b\langle z,E^{x,a}_3\rangle E^{x,a}_2+b\langle z,E^{x,a}_2\rangle E^{x,a}_3,\qquad z\in K^{x,a}
$$
where $b=\sqrt{r^2+a^2}$. Then $K$ is the orbit of $w^{x,a}$ and $Y(w^{x,a})=(w^{x,a})^\prime$ holds on $\Bbb R$.
\end{lemma}

\begin{proposition}\label{circles} Let $(K,Y)$ be an oriented circle in $\Bbb S^2$ and let $(p,\xi)\in M_r$ satisfy $p\notin K$. Then there exists $z\in K$ and an oriented circle $(T,B)$ in $\Bbb S^2$ such that $z,p\in T$, $B_z=Y_z$ and $B_p=\xi$.
\end{proposition}

\begin{proof}
Denote by $Q_z$ the vector space generated by $\{p-z,Y_z\}$ for $z\in K$. Since $p-z$ and $Y_z$ are linearly independent, $Q_z$ is two-dimensional. Now $T_z=(p+Q_z)\cap\Bbb S^2$ is a non-degenerate circle in $\Bbb S^2$ as it contains two distinct points $p,z\in\Bbb S^2$. Fixing $z\in K$, we are going to show that there exists a vector field $B$ of length $r$ on $T_z$ such that $B_z=Y_z$. For, if we define 
$$
R_z=r^2(p-z)-\langle p-z,Y_z\rangle Y_z,\qquad z\in K
$$
then $R_z\ne 0$ by linear independence of $\{p-z,Y_z\}$ and we can set $V_z=\frac{R_z}{|R_z|}$. So $\{V_z,\frac 1rY_z\}$ is an orthonormal basis in $Q_z$. Let $q_z$ be the orthogonal projection of $p$ onto $Q_z$ and define $p_z=p-q_z$, $\theta_z=\sqrt{1-|p_z|^2}$. So $T=(p_z+Q_z)\cap\Bbb S^2$. According to Remark \ref{somepropcirc}, 
$$
B_z(\tau)=\frac 1{\theta_z}\left[\langle\tau,Y_z\rangle V_z-\langle\tau,V_z\rangle Y_z\right],\qquad\tau\in T_z
$$
is a vector field of length $r$ on $T_z$. Since $z-p$ and $z-p_z$ belong to $Q_z$ and $p_z\perp Q_z$, we have $z=p_z+\langle z,V_z\rangle V_z+\frac 1{r^2}\langle z,Y_z\rangle Y_z=p_z+\langle z,V_z\rangle V_z$ as $z\perp Y_z$, hence
$$
1=|z|^2=|p_z|^2+\langle z,V_z\rangle^2,\qquad\theta_z=|\langle z,V_z\rangle|.
$$
But
$$
|R_z|\langle z,V_z\rangle=\langle z,R_z\rangle=r^2\langle z,p-z\rangle=r^2(\langle z,p\rangle-1)\le 0
$$
so we conclude that $\theta_z=-\langle z,V_z\rangle$. From this we obtain that $B_z(z)=-\frac 1{\theta_z}\langle z,V_z\rangle Y_z=Y_z$. Eventually, $B_z(p)=\frac 1{\theta_z}\left[\langle p,Y_z\rangle V_z-\langle p,V_z\rangle Y_z\right]$. It remains to prove that the mapping
$$
L:K\to\{\zeta\in T_p\Bbb S^2:|\zeta|=r\}:z\mapsto B_z(p)
$$
is a surjection. Since $K$ and $\{\zeta\in T_p\Bbb S^2:|\zeta|=r\}$ are homeomorphic with $\Bbb S^1$ and $L$ is continuous, it is sufficient to prove that $L$ is locally injective by Proposition \ref{surjection}. Here we can easily see that $L_z$ spans the one-dimensional vector space $Q_z\cap\{p\}^\perp$. 

So let us study injectivity of $L$. Let $K=(v+U)\cap\Bbb S^2$ where $U$ is a two-dimensional subspace in $\Bbb R^3$, $v\perp U$ and $|v|<1$. Let $z_1\in K$. Then there exists an orthonormal basis $u_1,u_2$ in $U$ such that $z_1=v+\xi u_1$ where $1=|v|^2+\xi^2$ and $Y(z_1)=ru_2$. If $z_2\in K$ satisfies $z_1\ne z_2$ then there exists a unique $\Delta\in(-\pi,\pi]\setminus\{0\}$  such that 
$$
z_2=v+\xi u_1\cos\Delta+\xi u_2\sin\Delta
$$
and, from this, 
$$
Y(z_2)=r[-u_1\sin\Delta+u_2\cos\Delta].
$$
Then $Q_{z_1}\cap Q_{z_2}$ is a one-dimensional space spanned by 
$$
A=(z_1-p)\sin\Delta+\frac{\xi}{r}(1-\cos\Delta)Y(z_1)=(z_2-p)\sin\Delta-\frac{\xi}{r}(1-\cos\Delta)Y(z_2).
$$
Obviously, the vector $A$ belongs also to $\{p\}^\perp$ iff
\begin{equation}\label{rum}
\psi(\Delta):=\frac{\sin\Delta}{1-\cos\Delta}=\frac{\xi\langle p,u_2\rangle}{1-\langle p,z_1\rangle}.
\end{equation}
Now $\psi:(-\pi,\pi]\setminus\{0\}\to\Bbb R$ is a bijection and the right hand side of \eqref{rum} is bounded by a constant $C_{p,K}$ irrespective of $z_1$, $z_2$, $u_1$ or $u_2$, as $p\notin K$. So $\Delta$ satisfying the identity \eqref{rum} must verify to $|\Delta|\ge\varepsilon_{p,K}>0$ and, consequently, $|z_1-z_2|\ge\varepsilon^\prime_{p,K}>0$. In particular, $L$ is locally injective and, consequently, $L$ is surjective. The identity \eqref{rum} then also implies that 
$$
\{z\in K\setminus\{z_1\}:L(z)\in\{-L(z_1),L(z_1)\}\}=\{z\in K\setminus\{z_1\}:\operatorname{dim}\,Q_{z_1}\cap Q_{z}\cap\{p\}^\perp=1\}
$$
contains exactly one element $z_2$, which, by surjectivity of $L$, must satisfy $L(z_1)=-L(z_2)$. In particular, $L$ is injective.
\end{proof}

\subsection{Proof of Theorem \ref{supi}}

Let $(p_1,\xi_1),(p_3,\xi_3)\in M_r$. We are going to show that, choosing a suitable piece-wise constant control  $a$ in the equation \eqref{adithree}, we can reach $(p_3,\xi_3)$ from $(p_1,\xi_1)$ by the solution \eqref{adithree} with this control $a$ in any time $T\ge 2\pi/r$. We are going to proceed in steps.

First let $a_1=0$ and move $(p_1,\xi_1)$ along the solution of \eqref{adithree} with the constant control $a_1$ to some $(p_2,\xi_2)$ in a very short time just to arrange $p_2\ne p_3$.

Next let $a_2$ be an extremely large constant control so that the orbit $K_2$ of the solution $w^{(p_2,\xi_2),a_2}$ does not contain $p_3$. This can be done by choosing a large control $a$ as the diameter of the orbit is $2r/\sqrt{r^2+a^2}$ by \eqref{wsol}. This solution defines an oriented circle $(K_2,Y_2)$ in $\Bbb S^2$ and $p_3\notin K_2$. Hence, by Proposition \ref{circles}, there exists an oriented circle $(K_3,Y_3)$ in $\Bbb S^2$ such that $z\in K_2\cap K_3$, $p_3\in K_3$, $Y_2(z)=Y_3(z)$ and $Y_3(p_3)=\xi_3$. This circle $K_3$ is associated to a control $a_3\in\Bbb R$.

Let $a$ be the piece-wise constant control with steps $a_1$, $a_2$ and $a_3$ at times $\tau_1$, $\tau_2$ and $\tau_3$ so that the solution $X$ to \eqref{adithree} with this control satisfies $X(0)=(p_1,\xi_1)$, $X(\tau_1)=(p_2,\xi_2)$, $X(\tau_2)=(z,Y_2(z))=(z,Y_3(z))$ and $X(\tau_3)=(p_3,\xi_3)$. Now $\tau_1$ was as small as we wanted, $\tau_2-\tau_1$ too because $a_2$ was large and the periodicity of the solutions to \eqref{adithree} with a constant control $a$ is $2\pi/\sqrt{r^2+a^2}$ by \eqref{wsol}. Hence $\tau_3-\tau_2$ is not larger that $2\pi/r$ since we do not let the solution run the full period. Altogether, $\tau_3<T$.

Let $a_4\in\Bbb R$ be a control such that $T-\tau_3\in\{2\pi k/\sqrt{r^2+a_4^2}:k\ge 0\}$ and let $a=a_4$ on $(\tau_3,T]$. Then $X(T)=X(\tau_3)=(p_3,\xi_3)$. In other words, we let the solution revolve to wait for the time $T$, to wind up in the point of the departure $(p_3,\xi_3)$.

\section{Exponential ergodicity in $M_r$, $r>0$}

In this section, again, we consider the Markov semigroup $(P_t)$ and its adjoint semigroup $(P^*_t)$ restricted to $M_r$ as in Convention \ref{restrmr}, with $r>0$ fixed. We are going to prove the exponential convergence to the invariant measure $\lambda_r$ in total variation via the Doeblin theorem and a minorization condition due to \cite{meytwe} and \cite{masthi}.

\begin{lemma}\label{bdensity} The transition densities satisfy $p>0$ on $(2\pi/r,\infty)\times M_r\times M_r$.
\end{lemma}

\begin{proof}
We develop the idea of \cite[Section 5.2]{meytwe} and the proof of \cite[Lemma 2.3]{masthi}. According to Theorem \ref{smoothdens}, the transition densities $p(t,x,\cdot)$ are smooth in all three variables. Let $t_1>2\pi/r$ and $t_2>0$ satisfy $t=t_1+t_2$. Let also $x_0,y_0\in M_r$ be such that $p(t_2,\cdot,\cdot)\ge\varepsilon$ on a neighbourhood $O_{x_0}\times O_{y_0}$ for some $\varepsilon>0$. Then, form the Chapman-Kolmogorov identity
$$
p(t,x,y)=\int_{M_r}p(t_1,x,z)p(t_2,z,y)\,d\lambda_r(z)
\ge\varepsilon p(t_1,x,O_{x_0})>0,\qquad\forall x\in M_r,\,\,\forall y\in O_{y_0}
$$
since the support of $p_{t_1,x}$ is $M_r$ by Theorem \ref{supi}. Now if $p(t,x_1,y_1)=0$ for some $x_1,y_1\in M_r$, let $Q\in\Bbb R^3\otimes\Bbb R^3$ be one of the two unitary matrices for which $\widetilde Q=Q\otimes Q=\operatorname{diag}\,[Q,Q]$ satisfies $\widetilde Qy_1=y_0$. Then $0=p(t,x_1,y_1)=p(t,\widetilde Qx_1,y_0)$ by Corollary \ref{rotations_dens}, which is a contradiction.
\end{proof}

\begin{theorem}\label{experg} There exist positive constants $c_r,\alpha_r$ such that 
\begin{equation}\label{tvconv}
\|P^*_t\nu-\lambda_r\|\le c_re^{-\alpha_rt}\|\nu-\lambda_r\|,\qquad\forall t\ge 0
\end{equation}
holds for every probability measure $\nu$ on $\mathscr B(M_r)$, where $\|\cdot\|$ is the norm in total variation on $M_r$.
\end{theorem}

\begin{proof}
Set $\tau=4\pi/r$. According to Lemma  \ref{bdensity}, there exists $\varepsilon>0$ such that $p_{\tau,x}(A)\ge\varepsilon\lambda_r(A)$ holds for every $x\in M_r$ and every $A\in\mathscr B(M_r)$. Hence, by the Doeblin theorem\footnote{See e.g. \cite[Theorem 4]{diafre} for a particularly simple proof of the Doeblin theorem.}, $(P^*_t)$ has a unique invariant probability measure $\mu$ on $M_r$ and there exist positive constants $c_r$ and $\alpha_r$ such that
$$
\|P^*_t\nu-\mu\|\le c_re^{-\alpha_rt}\|\nu-\mu\|,\qquad\forall t\ge 0
$$
holds for every probability measure $\nu$ on $\mathscr B(M_r)$. But $\lambda_r$ is the unique invariant probability measure on $M_r$ by Theorem \ref{prop-horem}.
\end{proof}

\section{Invariant measures and attractivity on $T\Bbb S^2$}

In this last section, we are going to study the global dynamics on the full target space $T\Bbb S^2$. We will identify the set of invariant probability measures on $\mathscr B(T\Bbb S^2)$, the set of ergodic probability measures on $\mathscr B(T\Bbb S^2)$ and it will be shown that the dual Markov semigroup is always attractive.

\begin{definition} Extend $\lambda_r$ from $\mathscr B(M_r)$ to $\mathscr B(T\Bbb S^2)$, in the unique way to obtain a probability measure on $\mathscr B(T\Bbb S^2)$, i.e. $A\mapsto\lambda_r(A\cap M_r)$. Let us denote this extension still by $\lambda_r$.
\end{definition}

\begin{definition} If $\nu$ is a probability measure on $T\mathbb{S}^2$, we define the probability measures
$$
\nu_\ast(U)=\nu\,\{(p,\xi)\in T\mathbb{S}^2:|\xi|\in U\},\qquad U\in\mathscr B([0,\infty))
$$
$$
\bar\nu(A)=\nu\,(A\cap M_0)+\int_{(0,\infty)}\lambda_r(A\cap M_r)\,d\nu_\ast,\qquad A\in\mathscr B(T\mathbb{S}^2)
$$
in the notation of \eqref{submanem}.
\end{definition}

\begin{remark} One can check by the definition of $\lambda_r$ that the mapping $r\mapsto\lambda_r(A\cap M_r)$ is Borel measurable on $(0,\infty)$ for every $A\in\mathscr B(T\mathbb{S}^2)$ by the Fubini theorem.
\end{remark}

\begin{theorem}
Let $z$ be a solution of \eqref{gsde} on $T\mathbb{S}^2$ with an initial distribution $\nu$ on $\mathscr B(T\mathbb{S}^2)$. Then the laws of $z(t)$ converge in total variation on $T\mathbb{S}^2$ to $\bar\nu$ as $t\to\infty$. Moreover, $\nu$ is invariant for \eqref{gsde} iff $\nu=\bar\nu$ and $\{\delta_x,\lambda_r:x\in M_0,\,r>0\}$ is the set of ergodic probability measures for \eqref{gsde}.
\end{theorem}

\begin{proof} Let $F:[0,\infty)\times\mathscr B(T\Bbb S^2)\to[0,1]$ be a regular version of a conditional probability measure $\nu(\cdot||\xi|=r)$ on $\mathscr B(T\Bbb S^2)$ for $r\ge 0$, i.e. $F(r,\cdot)$ is a probability measure on $\mathscr B(T\Bbb S^2)$ for every $r\ge 0$, $F(\cdot,A)$ is Borel measurable on $[0,\infty)$ for every $A\in\mathscr B(T\Bbb S^2)$ and 
\begin{equation}\label{disint}
\nu(A\cap\{(p,\xi):|\xi|\in U\})=\int_U F(r,A)\,d\nu_*(r)
\end{equation}
holds for every $A\in\mathscr B(T\Bbb S^2)$ and $U\in\mathscr B[0,\infty)$. The equality \eqref{disint} implies that
\begin{equation}\label{disint2}
\int_{T\Bbb S^2}h(|\xi|,p,\xi)\,d\nu(p,\xi)=\int_{[0,\infty)}\left(\int_{T\Bbb S^2}h(r,y)\,dF_r(y)\right)\,d\nu_*(r)
\end{equation}
holds for every bounded measurable $h:[0,\infty)\times T\Bbb S^2\to\Bbb R$. In particular, setting $h(r,p,\xi)=\mathbf 1_{[r=|\xi|]}$, we obtain that $\nu_*(O)=1$ where $O=\{r\ge 0:F(r,M_r)=1\}$. So \eqref{disint2} implies that
\begin{eqnarray*}
(P^*_t\nu)(A)&=&\int_{T\Bbb S^2}p(t,x,A)\,d\nu=\int_O\left(\int_{M_r}p(t,x,A)\,dF_r(x)\right)\,d\nu_*(r)
=\int_O(P^*_tF_r)(A\cap M_r)\,d\nu_*(r)
\\
&=&\nu(A\cap M_0)+\int_{O\cap(0,\infty)}(P^*_tF_r)(A\cap M_r)\,d\nu_*(r)
\end{eqnarray*}
holds for every $t\ge 0$ and $A\in\mathscr B(T\Bbb S^2)$. By a contradiction argument, we get that $P^*_t\nu$ converge in total variation on $T\mathbb{S}^2$ to $\bar\nu$, by Theorem \ref{experg}.

To prove the invariance part of the claim, realize that
$$
\int_{T\Bbb S^2}h\,d\bar\nu=\int_{M_0}h\,d\nu+\int_{(0,\infty)}\left(\int_{M_r}h\,d\lambda_r\right)\,d\nu_*
$$
holds for every bounded measurable $h:T\Bbb S^2\to\Bbb R$ by the definition of the measure $\bar\nu$. Hence, setting $h(x)=p(t,x,A)$, we get that
$$
(P^*_t\bar\nu)(A)=\nu(A\cap M_0)+\int_{(0,\infty)}\lambda_r(A\cap M_r)\,d\nu_*=\bar\nu(A)
$$
holds for every $A\in\mathscr B(T\Bbb S^2)$ by Theorem \ref{prop-horem}. In particular, $\bar\nu$ is invariant. If $\nu$ is invariant then $\nu=\lim_{t\to\infty}P^*_t\nu=\bar\nu$ by the first part of the proof.

Concerning the ergodic measures (according to Definition \ref{basicdef}), the probability measures $\{\delta_x,\lambda_r:x\in M_0,\,r>0\}$ are invariant by the second part of the proof and ergodicity follows from Remark \ref{erg} as ergodic probability measures are the extremal points of the set of all invariant probability measures (see e.g. Proposition 3.2.7 in \cite{daza}). Indeed, the probability measure $\nu_a$ is ergodic for \eqref{gsde} iff $a$ is an extremal point in the convex set of probability measures on $\mathscr B(M_0\dot\cup(0,\infty))$. This occurs iff $a$ is a Dirac measure, i.e. either $a=\delta_x$ for some $x\in M_0$ (hence $\nu_a=\delta_x$) or $a=\delta_r$ for some $r>0$ (hence $\nu_a=\lambda_r$).
\end{proof}

\begin{remark}\label{erg} Invariant measures for \eqref{gsde} can be uniquely described as measures
$$
\nu_a(A)=a\,(A\cap M_0)+\int_{(0,\infty)}\lambda_r(A\cap M_r)\,da,\qquad A\in\mathscr B(T\mathbb{S}^2)
$$
where $a$ is a Borel probability measure on the Polish space\footnote{Topological spaces that can be metrized by a complete separable metric are called Polish spaces.} $X=M_0\dot\cup(0,\infty)$, i.e. $G\subseteq X$ is open iff $G\cap M_0$ is open in $M_0$ and $G\cap(0,\infty)$ is open in $(0,\infty)$. $X$ is Polish as so are $M_0$ and $(0,\infty)$. The assignment $a\mapsto\nu_a$ is a bijection onto the set of invariant probability measures.
\end{remark}

\appendix

\section{Lie algebra}

Let $U$ be an open set on a $C^\infty$-manifold.

\begin{itemize}
\item The set $\mathscr L$ of all smooth tangent vector fields on $U$  is  a vector space with the Jacobi bracket.  Any  vector subspace of $\mathscr L$ closed under the Jacobi bracket is called \textit{ a Lie algebra}.
\item If $\mathcal X$ is a set of smooth tangent vector fields on $U$, then we denote by $\mathscr L(\mathcal X)$ the smallest Lie algebra containing $\mathcal X$.
\item If $\mathcal A\subseteq\mathscr L$ and $p\in U$, then we define $\mathcal A(p)=\{A_p:A\in\mathcal A\}$.
\end{itemize}

\begin{proposition} Define $L_0=\operatorname{span}\{\mathcal X\}$ and $L_n=\operatorname{span}\{L_{n-1}\cup\{[A,B]:A,B\in L_{n-1}\}\}$. Then $\bigcup L_n=\mathscr L(\mathcal X)$.
\end{proposition}

\begin{proposition}\label{dimeq} Let $X_1,\dots,X_m,Y\in\mathscr L$ and let $f_i\in C^\infty(U)$. Then
$$
\mathscr L(X_1,\dots,X_m,Y)(p)=\mathscr L(X_1,\dots,X_m,Y+\sum_{j=1}^mf_jX_j)(p),\qquad p\in U.
$$
\end{proposition}

\begin{proof} Let us write $\mathcal A^1=\{X_1,\dots,X_m,Y\}$, $\mathcal A^2=\{X_1,\dots,X_m,Y+\sum_{j=1}^mf_jX_j\}$,
$$
\mathscr C^i=\left\{\sum_{k=1}^Kh_kL_k:h_k\in C^\infty(U),L_k\in\mathscr L(\mathcal A^i), K\in\mathbb{N}\right\}.
$$
Apparently, $\mathscr C^i$ is a Lie algebra for $i\in\{1,2\}$, $\mathcal A^i\subseteq\mathscr C^j$ whenever $\{i,j\}=\{1,2\}$ hence $\mathscr L(\mathcal A^i)\subseteq\mathscr C^j$ whenever $\{i,j\}=\{1,2\}$. But then
$$
\mathscr L(\mathcal A^i)(p)\subseteq\mathscr C^j(p)\subseteq\mathscr L(\mathcal A^j)(p).
$$
\end{proof}

\begin{theorem}[H\"ormander]\label{hormanthm}
Let $M$ be a Riemannian manifold with a countable topological basis, let $X^1,\dots,X^m,Y$ be smooth vector fields on $M$, let $Z$ be a smooth funciton on $M$ and let $\mu$ be a Radon measure on $\mathscr B(M)$ such that
\begin{equation}\label{hormmeasid}
\int_M\left\{Zh+Y(h)+\sum_{i=1}^mX^i(X^i(h))\right\}\,d\mu=0,\qquad\forall h\in C^\infty_{comp}(M)
\end{equation}
and
$$
\operatorname{span}\{L_p:L\in\mathscr L(X^1,\dots,X^m,Y)\}=T_pM,\qquad\forall p\in M.
$$
Then $\mu$ has a $C^\infty$-smooth density with respect to the Riemannian measure on $M$.
\end{theorem}

\begin{proof}
Let $\varphi:O\to U$ be a diffeomorphism from an open set $O\subseteq\Bbb R^d$ onto an open set $U\subseteq M$, denote by $\phi$ the inverse of $\varphi$, define $\theta(A)=\nu\,[\varphi[A]]$ for $A\in\mathscr B(O)$, decompose $X^i_\varphi=\sum_{i=1}^dx_i\partial^i_\varphi$, $Y_\varphi=\sum_{i=1}^dy_i\partial^i_\varphi$ on $O$ and define $z=Z(\varphi)$ and
$$
Q=-y+2\sum_{i=1}^m(\operatorname{div}x^i)x^i,\qquad S=z-\operatorname{div}y+\sum_{i=1}^m\operatorname{div}[(\operatorname{div}x^i)x^i].
$$
Then \eqref{hormmeasid} implies that
$$
S\theta+Q(\theta)+\sum_{i=1}^mx^i(x^i(\theta))=0
$$
holds in the sense of distributions on $O$. According to Proposition \ref{dimeq}, 
$$
\operatorname{span}\{L_z:L\in\mathscr L(x^1,\dots,x^m,y)\}=\operatorname{span}\{L_z:L\in\mathscr L(x^1,\dots,x^m,Q)\}=\Bbb R^d,\qquad\forall z\in O.
$$
Hence, by the H\"ormander theorem \cite{hormanderl}, $\theta$ is absolutely continuous with respect to the Lebesgue measure and the density $\rho$ belongs to $C^\infty(O)$. If we define $L=\sqrt{\operatorname{det}g_{ij}}$ on $U$ then
$$
\nu(B)=\int_O\mathbf 1_B(\varphi)\rho\,dx=\int_B\frac{\rho(\phi)}L\,dx,\qquad B\in\mathscr B(U).
$$
By a localization argument, we obtain that $\nu$ has a density $R\in C^\infty(M)$ with respect to $dx$.
\end{proof}

\section{Density of product functions}

\begin{proposition}\label{density} Let $M$ be a compact submanifold in $\Bbb R^m$. Then 
$$
\mathcal P=\operatorname{span}\,\{h_1(t)h_2(x)h_3(z):h_1\in C^\infty_{comp}(0,\infty),\,h_2,h_3\in C^\infty(M)\}
$$
is dense in the space $C^\infty_{comp}((0,\infty)\times M\times M)$ in the following sense. Let $h\in C^\infty_{comp}((0,\infty)\times M\times M)$. Then there exist $\chi_n\in\mathcal P$ such that
$$
\chi_n\to h\quad\textrm{and}\quad X_m\dots X_1\chi_n\to X_m\dots X_1h
$$
uniformly on $(0,\infty)\times M\times M$ for every vector fields $X_1,\dots,X_m$ on $(0,\infty)\times M\times M$.
\end{proposition}

\begin{proof}
Let $0<a<b$ be such that the support of $h$ is contained in $(a,b)\times M\times M$ and extend $h$ to a smooth compactly supported function in $\Bbb R\times\Bbb R^m\times\Bbb R^m$. This can be done by standard methods of local extensions and a partition of unity as $M$ is assumed to be compact. Denote by $h_1$ such an extension. The support of $h_1$ fits in a some large cube $Q=(-N,N)^{1+m+m}$ and we can replicate $h_1$ to each cube $2Nk+Q$ for $k\in\Bbb Z^{1+m+m}$ to obtain a smooth $2N$-periodic function $h_2$ such that $h_1=h_2$ in $Q$. Now we can apply the Fej\'er's theorem on Fourier series to find a sequence of functions 
$$
\xi_n\in\operatorname{span}\,\{v_1(t)v_2(x)v_3(z):v_1\in C^\infty_{2N\textrm{-per}}(0,\infty),\,h_2,h_3\in C^\infty_{2N\textrm{-per}}(\Bbb R^m)\}
$$
such that $\xi_n\to h_2$ in $C^\infty(\Bbb R^{1+m+m})$. If $\rho\in C^\infty(\Bbb R)$ has support in $(0,\infty)$ and $\rho=1$ on $[a,b]$ then we can define $\chi_n(t,x,z)=\rho(t)\xi_n(t,x,z)$. The restrictions of $\chi_n$ to $(0,\infty)\times M\times M$ belong to $\mathcal P$ and approximate $h$ in the asserted sense.
\end{proof}

\section{Continuous surjections between circles}

\begin{proposition}\label{surjection} Let $f:\Bbb S^1\to\Bbb S^1$ be continuous and locally injective. Then $f$ is a surjection.
\end{proposition}

\begin{proof}
Since $\Bbb S^1$ is compact and $f$ is continuous, $f[\Bbb S^1]$ is also a compact. But local injectivity of $f$ implies that $f[\Bbb S^1]$ is open. Hence $f$ is a surjection as $\Bbb S^1$ is connected.
\end{proof}

\bibliographystyle{ams-pln}

\bibliography{literature}

\end{document}